\documentclass[final]{siamart1116}

\usepackage{times}
\usepackage{todonotes}
\usepackage{enumerate}
\usepackage{mathptmx,helvet,type1cm,xargs,srcltx,latexsym,bm,wrapfig,multicol,makeidx,stmaryrd}
\usepackage{amssymb,ifthen,amsfonts,bbm,graphicx,twoopt,framed,marvosym,phaistos}
\ifpdf
  \DeclareGraphicsExtensions{.eps,.pdf,.png,.jpg}
\else
  \DeclareGraphicsExtensions{.eps}
\fi

\numberwithin{theorem}{section}

\newcounter{St}

\newtheorem{step}[St]{Step}

\newcommand{\TheTitle}{Analysis of nonsmooth stochastic approximation: the differential inclusion approach}
\newcommand{\TheTitleShort}{Analysis of nonsmooth stochastic approximation: the D.I. approach}

\newcommand{\TheAuthors}{S.~Majewski, B.~Miasojedow, E.~Moulines}

\headers{\TheTitleShort}{\TheAuthors}

\title{{\TheTitle}
\thanks
{
{This work was funded by the project BayesScale - Pierre Laffitte -, the Data Science Initiative of Ecole Polytechnique and by Polish National Science Center grant no. 2015/17/D/ST1/01198.}}
}

\author{
  Szymon Majewski\thanks{Institute of Mathematics, Polish Academy of Sciences
    (\email{smajewski@impan.pl})}.
  \and
  Blazej Miasojedow\thanks{Institute of Applied Mathematics and Mechanics, University of Warsaw, Institute of Mathematics, Polish Academy of Sciences, (\email{bmia@mimuw.edu.pl})}
  \and
  Eric Moulines \thanks{Ecole Polytechnique, Paris, National Research University Higher School of Economics (HSE), Moscow,
    (\email{eric.moulines@polytechnique.edu}).}
}

\newsiamremark{remark}{Remark}
\newsiamremark{example}{Example}

\newcommandx\sequence[3][2=,3=]
{\ifthenelse{\equal{#3}{}}{\ensuremath{\{ #1_{#2}\}}}{\ensuremath{\{ #1_{#2}, \eqsp #2 \in #3 \}}}}
\newcommandx\sequencePar[3][2=,3=]
{\ifthenelse{\equal{#3}{}}{\ensuremath{\{ #1({#2})\}}}{\ensuremath{\{ #1({#2}), \eqsp #2 \in #3 \}}}}

\newcommandx\subsequence[4][2=,3=,4=]
{\ifthenelse{\equal{#4}{}}{\ensuremath{\{ #1_{#2_{#3}}\}}}{\ensuremath{\{ #1_{#2_{#3}}, \eqsp #3 \in #4 \}}}}

\def\rset{\ensuremath{\mathbb{R}}}

\def\nset{\ensuremath{\mathbb{N}}}

\newcommand{\eqsp}{\;}

\newcommandx\vectornorm[3][1=,3=]{\ifthenelse{\equal{#1}{}}{\left\| #2 \right\|}{\left\| #2 \right\|_{#3}^{#1}}}  
\newcommandx\supnorm[2][1=]{\vectornorm{#2}^{#1}_\infty}
\newcommandx\lnorm[3][1=]{\left\lVert #2 \right\rVert^{#1}_{#3}}

\newcounter{hypA}
\newenvironment{hypA}{\begin{sf}\refstepcounter{hypA}\begin{itemize}
  \item[({\bf A\arabic{hypA}})]}{\end{itemize}\end{sf}}

\newcounter{hypB}
\newenvironment{hypB}{\begin{sf}\refstepcounter{hypB}\begin{itemize}
  \item[({\bf B\arabic{hypB}})]}{\end{itemize}\end{sf}}

\newcounter{hypP}
\newenvironment{hypP}{\begin{sf}\refstepcounter{hypP}\begin{itemize}
  \item[({\bf P\arabic{hypP}})]}{\end{itemize}\end{sf}}

\newcounter{hypPp}
\newenvironment{hypPp}{\begin{sf}\refstepcounter{hypPp}\begin{itemize}
  \item[({\bf {\^{P}}\arabic{hypPp}})]}{\end{itemize}\end{sf}}

\newcommand{\Vnorm}[1]{\left\Vert #1 \right\Vert}
\def\ie{i.e.}
\def\wrt{with respect to}

\newcommand{\1}{\ensuremath{\mathbbm{1}}}

\newcommandx{\indi}[2][1=]{\1^{#1}_{#2}}

\def\Xset{\mathsf{X}}

\def\Zset{\mathsf{Z}}
\def\Zsigma{\mathcal{Z}}

\newcommand{\ensemble}[2]{\left\{#1\,:\eqsp #2\right\}}
\newcommand{\set}[2]{\ensemble{#1}{#2}}

\newcommand{\pscal}[2]{{\langle #1 , #2 \rangle}}

  \newcommand{\coint}[1]{\left[#1\right)}
  \newcommand{\ocint}[1]{\left(#1\right]}
  \newcommand{\ooint}[1]{\left(#1\right)}
  \newcommand{\ccint}[1]{\left[#1\right]}

\newcommand{\closure}[1]{\overline{#1}}
\newcommand{\eqdef}{:=}
\def\rmd{\mathrm{d}}
\def\dist{\operatorname{d}}
\def\ball{\operatorname{B}}
\def\cvxind{\mathbb{I}}
\DeclareMathOperator*{\argmin}{arg\,min}

\DeclareMathOperator{\prox}{prox}

\newcommand{\lip}[1]{\Vert #1 \Vert_{\operatorname{Lip}}}
\newcommand{\loclip}[2]{\Vert #1 \Vert_{\operatorname{Lip},#2}}
\newcommand{\tangentcone}[1]{\operatorname{T}_{#1}}
\newcommand{\normalcone}[1]{\operatorname{N}_{#1}}
\newcommand{\PE}{\mathbb{E}}
\newcommand{\PP}{\mathbb{P}}
\def\iid{i.i.d.}
\newcommand{\sign}[1]{\operatorname{sign}(#1)}
\newcommand{\pospart}[1]{(#1)_+}
\newcommand{\Xb}[2]{z_{#1}^{(#2)}}
\def\mcf{\mathcal{F}}
\newcommand{\CPE}[3][]
{\ifthenelse{\equal{#1}{}}{{\mathbb E}\left[\left. #2 \, \right| #3 \right]}{{\mathbb E}_{#1}\left[\left. #2 \, \right | #3 \right]}}
\def\proxSGD{ProxSGD}
\newcommand{\normW}[2]{\left|#1\right|_{#2}}
\newcommand{\normWm}[2]{\left\|#1\right\|_{#2}}

\newcommand{\newproof}[1]{
  \theoremstyle{nonumberplain}
\theoremheaderfont{\normalfont\itshape}
\theorembodyfont{\normalfont}
\theoremseparator{.}
\theoremsymbol{\proofbox}
\newtheorem{#1}{Proof}
}
\newproof{proof1}

\begin{document}

\maketitle

\begin{abstract}
In this paper we address the convergence of stochastic approximation when the functions to be minimized are not convex and nonsmooth. We show that the "mean-limit" approach to the convergence which leads, for smooth problems, to the ODE approach can be adapted to the non-smooth case. The limiting dynamical system may be shown to be, under appropriate assumption, a differential inclusion. Our results expand  earlier works in this direction by \cite{benaim2005stochastic} and provide a general framework for proving  convergence for unconstrained and constrained stochastic approximation problems, with either explicit or implicit updates. In particular, our results allow us to establish the convergence of stochastic subgradient and proximal stochastic gradient descent algorithms arising in a large class of deep learning and high-dimensional statistical inference with sparsity inducing penalties.
\end{abstract}


\section{Introduction}
Stochastic approximation algorithms are stochastic processes defined iteratively as
\begin{equation}
\label{eq:general-SAA}
x_k = x_{k-1} + \gamma_k Y_k
\end{equation}
where $x_k$ takes value in $\rset^d$, $\sequence{Y}[k][\nset]$ is a sequence of random variable and $\sequence{\gamma}[k][\nset^*]$ is a sequence of stepsizes satisfying $\gamma_k > 0$, $\lim_{k \to \infty} \gamma_k= 0$ and $\sum_{k=1}^\infty \gamma_k = \infty$. The decreasing stepsizes imply that the rate of change of the parameter decreases as $k$ goes to infinity, thus providing an implicit averaging.
The value $x_k$ might represent the current fit of a parameter or the state of a system and $Y_{k}= \Phi(x_{k-1},\xi_k)$ is a (measurable) function of the past fit $x_{k-1}$ of the parameter and a new information observed at time $k$. Such problems have been considered in the early work by \cite{robbins1951stochastic} and since found numerous applications, especially in machine learning and computational statistics.

A powerful method to analyze stochastic gradient algorithm, introduced in the early works by \cite{ljung:1977} and \cite{kushner1978stochastic}  is the ordinary differential equation (ODE) method.
 The ODE  method has led to an enormous literature; see for example \cite{benaim1999dynamics}, \cite{kushner2012stochastic} and the references therein. The ODE method can be informally summarized as follows: first we rewrite
 \begin{equation}
 \label{eq:definition-Y-eta}
 Y_k= F(x_{k-1}) + \eta_k
 \end{equation}
 where $F: \rset^d \to \rset^d$ is a locally-Lipschitz vector field defined by an appropriate averaging and $\eta_k= Y_k - F(x_{k-1})$. If $\sequence{\xi}[k][\nset]$ is an \iid\ sequence and $\PE[\vectornorm{\Phi(x,\xi)}] < \infty$ for all $x \in \rset^d$, one may take for example $F(x)= \PE[\Phi(x,\xi_1)]$ and $\eta_k= Y_k - F(x_{k-1})$ which then is a martingale increment sequence. The situation becomes more complex when the noise $\sequence{\xi}[k][\nset]$ is no longer \iid. Of particular importance is the case where the conditional  distribution of $\xi_k$ given the past $\set{(x_j,\xi_j)}{j \in \{0,\dots, k-1\}}$ is Markovian (see for example \cite{atchade2017Perturbed}).
 The ODE approach to analyze the asymptotic behavior of the sequence $\sequence{x}[k][\nset]$ is to consider them as approximated solutions of the ODE $\dot{x}= F(x)$.

In this paper, we are primarily interested by the application of stochastic approximation algorithms to minimize a function $f: \rset^d \to \rset$. If the function $f$ is differentiable and the gradient $\nabla f$ is known, a classical method to minimize $f$ consists  in performing a gradient descent. If only a noise corrupted $H_k$ version of the gradient $\nabla f(x_{k-1})$ is available, a popular algorithm is the stochastic gradient descent (SGD) algorithm whose iterations are given by
$x_k = x_{k-1} - \gamma_k H_k$.
In this context $F(x)= - \nabla f$ and the associated ODE is $\dot{x} = -\nabla f(x)$.

If the function $f$ is not differentiable, it is no longer possible to use the SGD algorithm.
However, if the function $f$ is locally Lipschitz, the Clarke generalized gradient $\bar\partial f(x)$ can still be defined
(see \Cref{def:clarke-gradient} and \cite[Section~1.2]{clarke1990optimization}). The Clarke generalized gradient $\bar\partial f(x)$  is a point-to-set map: for any $x \in \rset^d$, $\bar\partial f(x)$ is a
nonempty convex compact subset of $\rset^d$. When $f$ is continuously differentiable at $x$, $\bar \partial f(x)$ reduces to the singleton $\{\nabla f(x)\}$.
When $f$ is convex, then $\bar\partial f(x)$ coincides with the subdifferential of convex analysis.
As above, if the Clarke gradient cannot be computed but a noise corrupted version of a selection of $\bar\partial f(x_{k-1})$ is available, we may consider a generalization of the stochastic subgradient algorithm
\[
x_k = x_{k-1} - \gamma_k \{v_{k-1} + \eta_k\} \eqsp, \quad \text{where $v_{k-1} \in \bar\partial f(X_{k-1})$}
\]
which we sometimes denote more concisely
\begin{equation*}
\label{eq:recursion-DI}
x_k \in x_{k-1} - \gamma_k \{ \bar\partial f(x_{k-1}) + \eta_k \} \eqsp.
\end{equation*}
The classical stochastic approximation algorithm update rule is replaced by a stochastic recursive inclusion:
\begin{equation*}
\label{eq:definition-general-SDI}
x_k \in x_{k-1} + \gamma_k \{ F(x_{k-1}) + \eta_k \}
\end{equation*}
where $F$ is a point-to-set map and $\eta_k$ is defined in \eqref{eq:recursion-DI}.
Such algorithms play an important role in game theory, as illustrated in
\cite{Benaim2006Stochastic,Benaim2012Perturbations} where numerous examples of stochastic recursive inclusions are introduced.

\cite{benaim2005stochastic} have shown that the "mean-limit" approach  leading to the ODE method in the smooth case can be extended to the analysis of stochastic recursive inclusion. In this case, the limit ODE is replaced by a solution of the differential inclusion
\begin{equation*}
\dot{x} \in F(x) \eqsp,
\end{equation*}
\ie\ an absolutely continuous mapping $x: \rset \to \rset^d$ such that $\dot{x}(t) \in F(x(t))$ for almost every $t \in \rset$. Such differential inclusions play a key role in the analysis of nonsmooth dynamical systems; see for example \cite[Chapter~4]{clarke1990optimization}
or \cite[Section~2.1]{aubin2012differential} and the references therein.\footnote{Just before submitting this paper, \cite{Davis2018Stochastic} published an analysis of stochastic subgradient algorithms. We have developed our results completely independently; our work is based on results obtained earlier in \cite{benaim2005stochastic} and even if some statements are similar, the proofs in this paper are all original.}

In this paper, we will also consider proximal algorithms, which have become an important tool in  nonsmooth optimization problems; the literature in this field is also huge, see for example  \cite{Attouch2007,beck2009fast,nemirovski2009robust,parikh2014proximal,Bolte2014Proximal}).
Proximal algorithms  with stochastic updates  have been proposed and studied in recent years.
One such algorithm is Proximal Stochastic Gradient Descent (proxSGD), that optimizes composite convex function $P = f + g$ where $f$ is a continuously differentiable function with Lipschitz-gradients and $g$ is a "proximable" function (e.g.\ $g$ is lower semi-continuous and convex, but this notion can be extended to nonconvex functions). The proxSGD algorithm  alternates between stochastic gradient update for $f$ and deterministic proximal step for $g$. This above optimization problem plays a fundamental role in many machine learning problems, ranging from convex optimization such as convex regression problem with sparsity inducing penalties like LASSO to highly nonconvex problem such as optimizing the weights of deep neural networks.
Numerous papers have been devoted to the case when $f$ and $g$ are both convex, $f$ gradient Lipschitz and $g$ lower semi-continuous; see for example
\cite{rosasco2014convergence,nitanda2014stochastic,Xiao2014Proximal,combettes2016Stochastic}. Atchade et al.~\cite{atchade2017Perturbed} have extended these results in the  Markovian noise case.
In recent years, triggered by the surge deep learning, the nonconvex case has started to attract many research efforts, at least in the smooth case ($f$ gradient Lipschitz and $g \equiv 0$); see for example
\cite{ghadimi2013Stochastic,allen2016variance} and the references therein.
For the nonsmooth and nonconvex case, the results are still partial. Ghadimi et al. \cite{ghadimi2016Minibatch} considered the case where $f$ is
differentiable but possibly nonconvex  and $g$ is non-differentiable but convex. They have analyzed the deterministic proximal gradient algorithm (where the full gradient is
computed at each iteration). They have also extended their results to the stochastic case; Reddi et al. \cite{reddi2016stochastic}
provides rates of convergence.

We consider in this paper nonconvex and nonsmooth minimization problems. We establish  the convergence of stochastic inclusion equation generalizing \eqref{eq:definition-general-SDI} by allowing implicit steps and projections on a closed compact convex set at each iteration. Our results generalize \cite{benaim2005stochastic}.  We also discuss the stability of the limit differential inclusion by means of locally Lipschitz continuous and regular Lyapunov functions (see \Cref{def:regular-function}). We in particular establish a characterization of the possible limit point of the stochastic approximation algorithm as the set of zeros of an upper-bound of the set-valued Lie derivative (see \Cref{def:Lie-derivative}) of the
Lyapunov function. We then apply our results to the analysis of the proximal stochastic gradient descent for the composite minimization problem $P= f+g$, under assumptions  on the noise sequence analogous to those commonly used for the SGD in the smooth nonconvex case.   We also show  that $V= f + g$ can play the role of a Lyapunov function. We finally analyse a projected version of stochastic subgradient algorithm.

The paper is organized as follows. In \Cref{sec:assumptions-notations}, we introduce our main assumptions and notations and introduce the proximal stochastic gradient and projected subgradient algorithms.
In \Cref{sec:convergencePAD}, we state and prove  our main convergence results under the assumption that the iterates are stable.
In  \Cref{sec:convergencePPAD}, we extend these convergence results to the case where the updates are projected on a compact convex set.
In \Cref{sec:applications}, we consider  applications of our main results to {\proxSGD}  and projected stochastic subgradient. Finally in \Cref{ap:prelim} we present postponed proofs.

\section{Assumptions and Notations}
\label{sec:assumptions-notations}

In this section we introduce definitions and notations.

\begin{definition}[Perturbed approximate discretization (PAD) and projected perturbed approximate discretization (PPAD) ]\label{def_pad} Let $\Xset$ be an open subset of $\rset^d$ and $F$ be a set-valued function mapping each point $x \in \Xset$ to a set $F(x) \subset \Xset$. We say that the sequence $\sequence{x}[k][\nset] \subset \Xset$ is a Perturbed Approximate Discretization
 with noise  $\sequence{\eta}[k][\nset]$ and step sizes $\sequence{\gamma}[k][\nset]$ if an only if there exists  $\sequence{y}[k][\nset]$ such that
 \begin{equation}
 \label{def:pad}
 x_k\in x_{k-1}+\gamma_k\left\{F(y_{k})+\eta_k\right\} \eqsp.
\end{equation}
Let $K$ a compact convex set.
We say that the sequence $\sequence{x}[k][\nset]$ is a $K$-Projected Perturbed Approximate Discretization ($K$-PPAD)
 with  noise $\sequence{\eta}[k][\nset]$ and  step sizes $\sequence{\gamma}[k][\nset]$  if and only if there exists $\sequence{y}[k][\nset]$ such that
 \begin{equation}
  \label{eq:ppad}
  x_k \in \Pi_K \left(x_{k-1} + \gamma_k \{F(y_k) + \eta_k\} \right) \eqsp.
 \end{equation}
 \end{definition}
By convention for a convex closed set $K$ and a given set $A$, by $\Pi_K(A)$ we denote the  projection of $A$ onto $K$, defined as $\Pi_K(A)=\{\Pi_K(a)\eqsp, a\in A\} $.

\begin{remark}
 The condition $ x_k \in \Pi_K \left(x_{k-1} + \gamma_k (F(y_k) + \eta_k) \right)$ is satisfied if and only if  there
 exists  $v_k\in F(y_k)$ such that
\begin{equation}\label{def:proj_alternative}
\langle x_k - z , x_k - x_{k-1} - \gamma_k(v_k + \eta_k) \rangle \leq 0
\end{equation}
for all $z \in K$.
\end{remark}

\begin{example}[Bena\"im \emph{et al.} discretization]
\cite[Definition~III]{benaim2005stochastic} deal with sequence satisfying a recursion of the form
\[
x_{k} \in  x_{k-1} + \gamma_k \{ F(x_{k-1}) + \eta_k \} \eqsp.
\]
Such a sequence clearly is a PAD with noise $\sequence{\eta}[k][\nset]$, step sizes $\sequence{\gamma}[k][\nset]$. In such case $y_k= x_{k-1}$.
\end{example}

We now show that the PAD and $K$-PPAD formalism cover the proximal stochastic gradient descent (\proxSGD) and the stochastic (sub)gradient algorithms for nonsmooth and nonconvex minimization problems. First we introduce some additional definitions and notations.

\begin{definition}[Generalized directional derivative, after \protect{\cite[Chapter~2, Section~1]{clarke1990optimization}}] \label{def:gen-dir-deriv}
Let $f: \rset^d \rightarrow \rset$ be a locally Lipschitz function at $x_0$. 
The generalized directional derivative of $f$ at $x_0$ in the direction $h \in \rset^d$ is:
\[
f^0(x_0, h) = \lim_{\delta \downarrow 0^+} \sup_{\substack{\vectornorm{\tilde{h}} \leq \delta \\ 0 < \lambda < \delta}} \frac{f(x_0 + \tilde{h} + \lambda h) - f(x_0 + \tilde{h})}{\lambda}\eqsp.
\]
\end{definition}
\begin{remark} Contrary to the standard directional derivative, the generalized directional derivative $f^0(x_0, h)$  is always well defined  for  any interior point $x_0$ of  the domain
$\mathrm{Dom}(f)= \set{x \in \rset^{n}}{|f(x)| < \infty}$.
\end{remark}
It is shown in \cite[Proposition 2.1.1]{clarke1990optimization} that the generalized directional derivative at $x_0$, $h \mapsto f(x_0,h)$ is
positively homogeneous and subadditive and $|f^0(x_0, h)| \leq L \Vnorm{h}$, where $L$ is a Lipschitz constant on some neighborhood of $x_0$. Denote by $\pscal{\cdot}{\cdot}$ the scalar product on $\rset^d$.
\begin{definition}[Clarke generalized gradient] \label{def:clarke-gradient}
Let $f: \rset^d \rightarrow \rset$ be a function and $x_0$ a point in the interior of $\mathrm{Dom}(f)$.
If $f$ is locally Lipschitz function at $x_0$,
the Clarke generalized gradient of $f$ at $x_0$ is the set defined by:
\[
\overline{\partial}f(x_0) = \set{\zeta \in \rset^d}{f^0(x,h) \geq \pscal{h}{\zeta}, \, \quad \text{for all $h \in \rset^d$}} \eqsp.
\]
\end{definition}
Similarly to subgradient, the Clarke generalized gradient is a set-valued generalization of the gradient. In particular,
when function $f$ is continuously differentiable at some point $x_0$,
then we have $\overline{\partial}f(x_0) = \{ \nabla f (x_0) \}$. Furthermore, if the function $f$ is convex and locally Lipschitz and $x_0$ belongs to the interior of its domain,
then the  Clarke generalized gradient of $f$ coincides with the subgradient.

We say that the set-valued map $F: \Xset \to \rset^d$   is convex-compact
if for any point $x \in \Xset$ the set $F(x)$ is convex and compact. We say that the set-valued function $F$
is locally bounded if for any compact set $K \subset \Xset$, $\bigcup_{x \in K} F(x)$ is bounded.
We also define upper hemicontinuity.
\begin{definition}[Upper Hemicontinuity]\label{def:upper-hemicontinuity}
Let $\Xset$ be an open subset of $\rset^d$. A set-valued map $F\colon \Xset \rightarrow \mathcal{P}(\rset^d)$ is said to be upper
hemicontinuous at $x \in \Xset$, if and only if  $F(x)$ is nonempty, and for every open neighborhood  $U$ of $F(x)$, there exist an open neighbourhood $V$ of $x$, such that:
\[
 \left( \bigcup_{z \in V} F(z) \right) \subseteq U\eqsp.
\]
\end{definition}
It is shown in \cite[Propositions~2.1.2]{clarke1990optimization} that if $f$ is Lipschitz on a neighborhood of $x_0$ with Lipschitz constant $\loclip{f}{x_0}$  then  $\overline{\partial}f(x_0)$ is a non-empty set, compact, convex, and for any $u \in \bar\partial f(x_0)$, $\vectornorm{u} \leq \loclip{f}{x_0}$.
By \cite[Proposition~2.1.5]{clarke1990optimization}, $\bar\partial f$ is upper hemicontinuous at $x_0$. For any compact set $K \subset \Xset$, $\sup_{x \in K} \loclip{f}{x} < \infty$ showing that $\bar\partial f$ is  locally bounded.
For proofs of those results and additional properties of Clarke generalized gradient, we refer the reader to \cite{clarke1990optimization}.

\begin{definition}[Regular function] \label{def:regular-function}
Let $f$ be a function and $x_0 \in \mathrm{Dom}(f)$. Assume that $f$ is locally Lipschitz at $x_0$.
We say that $f$ is regular at $x_0$, if and only if  for any $h \in\rset^d$
\[
f^0(x_0,h)=\lim_{\lambda\downarrow 0^+}\frac{f(x_0+\lambda h)-f(x_0)}{\lambda}\eqsp,
\]
where $f^0$ is the generalized directional derivatives (see \Cref{def:gen-dir-deriv}).
\end{definition}
In  words, $f$ is regular at $x_0$ if the directional derivatives exist for all directions $d\in\rset^d$ and coincide with the generalized directional derivatives.
\begin{remark}
It may happen that the usual directional derivative exists,
but does not coincide with the generalized directional derivative.
The classical example is $f(x) = - |x|$. As shown in \cite[Proposition~2.3.6]{clarke1990optimization}
every convex locally Lipschitz function is regular.
The same property obviously holds for any continuously differentiable function.
Note finally that if the functions $f_1,\dots,f_p$ are regular at $x_0$, then for any nonnegative weights $\alpha_1,\dots,\alpha_p$,
$\sum_{i=1}^p \alpha_i f_i$ is also regular at $x_0$, see \cite[Proposition~2.3.6]{clarke1990optimization}.
\end{remark}
Now we are ready to discuss the proximal gradient descent and projected subgradient algorithms.
\begin{example}[Perturbed proximal gradient descent algorithm]
\label{exam:proximal}
Consider the problem of minimizing a composite function $P=f+g$ define on an open subset of $\Xset \subseteq\rset^d$
where $f$ is continuously differentiable and $g$ is locally Lipschitz, bounded from below and regular (see \Cref{def:regular-function}).
The proximal gradient algorithm is defined by the following recursion:
\[
 x_{k} \in \prox_{\gamma_{k}, g}(x_{k-1}-\gamma_{k}\nabla{f}(x_{k-1}))\eqsp, \quad k \geq 1 \eqsp,
\]
where $\sequence{\gamma}[k][\nset^*]$ is a sequence of stepsizes and $\prox_{\gamma,g}$ stands
for the proximal operator defined by
\begin{equation}\label{eq:proximal}
\prox_{\gamma,g}(x) \in \argmin_{y\in\Xset}\left\{g(y)+ (2\gamma)^{-1} \vectornorm{y-x}^2 \right\}\eqsp.
 \end{equation}
In our settings the function $g$ is lower bounded and under this condition the set appearing in the right-hand side of \eqref{eq:proximal} is nonempty. Necessary and sufficient conditions for this algorithm to be well-defined can be
found in  \cite[Excersise 1.24]{Rockafellar2009}.
In the perturbed version of the proximal gradient algorithm, for any $k\in\nset^*$ we replace the gradient
$\nabla f(x_{k-1})$ by  a noise corrupted version $\nabla f(x_{k-1})+\zeta_k$  which leads to the recursion
\[
 x_{k} \in \prox_{\gamma_{k} g}(x_{k-1}-\gamma_{k} \{ \nabla f(x_{k-1})+\zeta_k \})\eqsp.
\]
The characterization of the minimum by the Clarke generalized gradient (see \cite[Proposition 2.3.2]{clarke1990optimization}) yields
\[
 0 \in \gamma_{k}^{-1} (x_{k}-x_{k-1})+ \nabla f(x_{k-1}) + \zeta_k +\overline\partial g(x_{k})\eqsp.
\]
Setting, for any $k\in\nset^*$, $\eta_k= -\zeta_k + \nabla f(x_{k})-\nabla f(x_{k-1})$ we get that for any $k\in\nset^*$
\[
 x_{k}\in x_{k-1}+ \gamma_k\left[-\nabla f(x_{k})-\overline\partial g(x_{k}) +\eta_k\right]\eqsp.
\]
Therefore perturbed proximal gradient is a PAD in the sense of \Cref{def_pad} with $F=-\nabla f -\overline\partial g$, $y_k=x_k$,  noise sequence $\sequence{\eta}[k][\nset^*]$, 
and step sizes $\sequence{\gamma}[k][\nset^*]$.
\end{example}

\begin{example}[Projected stochastic (sub)gradient] We consider projected subgradient algorithm framework introduced in the convex case by \cite{nemirovski2009robust} to solve the constrained minimization problem $\argmin_{x\in K } f(x)$.
Let $K$ be a bounded closed convex set and $f\colon K\to\rset$ be a locally Lipshitz function. 
Assume that for each iteration an oracle  returns a perturbed element of the subgradient, \ie\   $H_k \eqdef \nu_k +\zeta_k$ with $\nu_k \in\partial f(x_{k-1})$ and with $\zeta_k$ a perturbation.

The projected stochastic subgradient algorithm generates iteratively the sequence
$\sequence{x}[k][\nset^*]$  as follows
\[
 x_{k}=\Pi_K(x_{k-1}-\gamma_{k} H_{k})\eqsp,
\]
where $\sequence{\gamma}[k][\nset^*]$ is a sequence of stepsizes and $\Pi_K$ is the projection on the set $K$. The projected subgradient algorithm is a $K$-PPAD with field
$F=-\partial f$, $y_{k}=x_{k-1}$ and noise $\eta_k= - \zeta_k$. 
\end{example}

We will analyse the convergence of PAD (see~\eqref{def:pad}) under the following assumptions:
\begin{hypA}
\label{ass:open}
$\Xset$ is an open subset of $\rset^d$ and $F : \Xset \rightarrow \mathcal{P}(\rset^d)$ is a set-valued map, that is upper hemicontinuous,
cf.~\Cref{def:upper-hemicontinuity}, convex-compact valued  and locally bounded.
\end{hypA}
\begin{hypA}
\label{ass:stepsize}
The sequence of step sizes $\sequence{\gamma}[k][\nset^*]$ satisfies $\gamma_k > 0$, $\sum_{k=0}^{\infty} \gamma_k = \infty$, and $\lim_{k \rightarrow \infty} \gamma_k = 0$.

\end{hypA}

\begin{hypA}
\label{ass:noise}
The perturbation sequence $\sequence{\eta}[k][\nset^*]$ can be decomposed as  $\eta_k = e_k + r_k$,
where $\sequence{e}[k][\nset^*]$ and $\sequence{r}[k][\nset^*]$ are two sequences satisfying $\lim_{k \rightarrow \infty}\vectornorm{r_k} = 0$, and $\sum_{k=1}^{\infty} \gamma_k e_k$ converges.
\end{hypA}
\begin{hypA}
 \label{ass: limit_y_x} The approximation sequence $\sequence{y}[k][\nset^*]$ belongs to $\Xset$ and satisfies
 \[
  \lim_{k\to\infty}\vectornorm{x_k-y_k}=0.
 \]

\end{hypA}

Condition (A\ref{ass:open}) is a rather mild regularity condition. Upper hemicontinuity replaces the continuity of the vector field which plays a key role in the classical theory of stochastic approximation \cite{kushner2012stochastic}.
The requirement for $F$ to be convex-compact valued and locally bounded might be less obvious, but this  assumption is commonly used in nonsmooth analysis. This is not a serious limitation for the minimization problems we have primarily in mind.

The assumptions  (A\ref{ass:stepsize}, A\ref{ass:noise}) are usual in stochastic approximation literature \cite{andrieu:moulines:2006,fort:moulines:priouret:2012}.
It is worth noting, that condition (A\ref{ass:noise}) allows perturbations sequences which have random and deterministic components
and hence our results can be used for proving almost sure convergence for \proxSGD\ for which the proximal operator
is computed numerically and is therefore inexact (although in our framework the deterministic noise should vanish asymptotically faster than step size).
The assumptions (A\ref{ass: limit_y_x}) allows to cover both explicit and implicit discretization of differential inclusions as illustrated in \Cref{exam:proximal}.

As in classical ODE method for stochastic approximation,  establishing convergence results first requires to show that
algorithm is stable in the sense that the sequence $\sequence{x}[k][\nset^*]$ remains in some compact set.
This issue is non-trivial even in the noiseless case and might be challenging to establish in the stochastic case.
One of possible solution to overcome this difficulty is to introduce a projection on convex compact set $K$. This case is considered in \Cref{sec:convergencePPAD}, however first
 in \Cref{sec:convergencePAD} we consider the standard version of stochastic approximation, where we establish the convergence  of PAD assuming that the sequence
 $\sequence{x}[k][\nset^*]$ remains in some compact set.

\section{Convergence of Perturbed Approximate Discretisation}
\label{sec:convergencePAD}
In this section, we state our main convergence results for PAD.
First in \Cref{thm:limit_compact} we show that a translated and interpolated version of the PAD
 converges to a solution of a differential inclusion. Further in \Cref{thm:convergence} we combine these
results with Lyapunov stability conditions to obtain convergence of the iterates to the set of stationary points of the differential inclusion.
\begin{definition}[{Solution of Differential Inclusion}] \label{def:diff-inclusion-solution}
Let $\Xset \subseteq \rset^d$ be an open subset, $F : \Xset\rightarrow \mathcal{P}(\rset^d)$ be a  set-valued map, and $I \subseteq \rset$ be an interval. A function $x : I \rightarrow \Xset$ is a solution of  the differential inclusion $\dot{x} \in F(x)$ if it is absolutely continuous and for almost every $t \in I$, $\dot x(t) \in F(x(t))$.
\end{definition}
Let us define the piecewise linear interpolation $X_0 : \rset \to \rset^d$ of the sequence $\sequence{x}[k][\nset]$  with positive stepsize  $\sequence{\gamma}[k][\nset]$:
\begin{equation}\label{eq:def_X}
X_0(t) =
\begin{cases}
x_0 & \text{if} \quad t \leq 0 \\
x_k \frac{(t - t_{k-1})}{\gamma_k} + x_{k-1} \frac{(t_{k} - t)}{\gamma_k} & \text{if} \quad t \in [t_{k-1},t_k] \eqsp, k \geq 1
\end{cases}
\end{equation}
where $t_0= 0$ and for $k \geq 1$,
\begin{equation}\label{eq:def_t}
t_k = t_{k-1} + \gamma_k= \sum_{i=1}^k \gamma_i \eqsp.
\end{equation}
Let $\sequence{s}[k][\nset^*]$ be an increasing sequence of positive real numbers, such that $\lim_{k \to \infty} s_k= \infty$. Let us define shifted linear interpolation of $\sequence{x}[k][\nset]$ by
\begin{equation}\label{eq:shift}
X_k(t) = X_0(t + s_k) \eqsp,\  t \geq 0 \eqsp.
\end{equation}
Consider the PAD sequence $\sequence{x}[k][\nset]$ defined in \eqref{def:pad}: for $k \in\nset^*$, 
\begin{equation}
\label{eq:PAD-1}
x_k = x_{k-1} + \gamma_k( v_k + e_k+r_k) \eqsp,   
\end{equation}
where $\sequence{e}[k][\nset]$ and $\sequence{r}[k][\nset]$ are defined in A\ref{ass:noise} and  $v_k \in F(y_k)$.
Let us define piecewise constant functions  $\hat v,\hat r,\hat e$ on $\coint{0,\infty}$ as follows
\begin{align*}
\hat v(t)=\sum_{k=1}^\infty v_k \indi{\coint{t_{k-1},t_k}}(t) \;,\hat r(t)=\sum_{k=1}^\infty r_k \indi{\coint{t_{k-1},t_k}}(t) \;, \hat e(t)&=\sum_{k=1}^\infty e_k \indi{\coint{t_{k-1},t_k}}(t) \eqsp,
\end{align*}
where $\sequence{t}[k][\nset^*]$ is defined in \eqref{eq:def_t}. Denote
\begin{equation}\label{def:hat_functions}
\hat V_0(t)=\int_0^t \hat v(s) \rmd s,\quad\hat R_0(t)=\int_0^t \hat r(s) \rmd  s,\quad \hat E_0(t)=\int_0^t \hat e(s) \rmd s\;.
\end{equation}
Analogously to \eqref{eq:shift}, for any $k\in\nset^*$ we denote by $\hat V_k$, $\hat E_k$, $\hat R_k$  the shifts of $\hat V_0, \hat E_0,\hat R_0$ respectively.
With this notation for any $k\in\nset^*$ we can decompose $X_k(t)$ as follows
\begin{equation}\label{eq:decompose_X}
X_k(t) = X_0(t + s_k) = x_0 + \hat{V}_k(t) + \hat{R}_k(t) + \hat {E}_k(t)\;.
\end{equation}
Without loss of generality we can assume that $\sum_{k=1}^{\infty} \gamma_k e_k = 0$ (if this is not true, we can just modify  $r_1$ and  $e_1$).
\begin{theorem}\label{thm:limit_compact}
  Let $\sequence{x}[k][\nset^*]$ be the PAD \eqref{eq:PAD-1}. Assume that conditions
  (A\ref{ass:open}--\ref{ass: limit_y_x}) hold, and there exists a
  compact set $K \subset \Xset$ such that $x_k \in K$ for any $k \geq 0$.
  Then
  \begin{enumerate}[(i)]
  \item the family of functions $\sequence{X}[k][\nset^*]$, defined by
    \eqref{eq:shift} is precompact in the topology of compact
    convergence: for any increasing sequence $\sequence{n}[k][\nset^*]$ of positive
    integers, there exist a subsequence
    $\sequence{\tilde n}[k][\nset^*]$ and an absolutely continuous
    function $X_\infty: \coint{0, +\infty} \to K$ such that for any
    $T>0$,
    $\lim_{k \to \infty}\sup_{t \in \ccint{0,T}} \vectornorm{X_{\tilde
        n_k}(t) - X_\infty(t)} = 0$ and $X_\infty$ is a solution of differential inclusion
         \[
   \dot{x} \in F(x)\eqsp.
    \]
  \item In addition, for any $t \geq 0$, $X_\infty(t)$ is a limiting
    point of the sequence $\sequence{x}[k][\nset^*]$.
  \end{enumerate}
\end{theorem}
\begin{proof1}
 We divide the proof into three steps.

\begin{step}
 \label{lem:equicontinuity}
   The family of functions $\sequence{X}[k][\nset^*]$, defined by
    \eqref{eq:shift} is precompact in the topology of compact
    convergence.  
 \end{step}
\begin{proof}
First we prove equicontinuity, cf. \Cref{def:pointwise-equicont}, of the sequence of functions $\sequence{X}[k][\nset]$.
The proof follows essentially by the same arguments as in \cite[Theorem 2.3.1]{kushner2012stochastic}, but we included it for completeness.
Obviously equicontinuity (see \Cref{def:pointwise-equicont})  of all terms on the RHS of \eqref{eq:decompose_X} implies
pointwise equicontinuity of the family of functions $\sequence{X}[k][\nset^*]$. Assumption (A\ref{ass:noise}) implies boundedness of the sequence $\sequence{r}[k][\nset^*]$.
Therefore the sequence of functions $\sequence{\hat R}[k][\nset^*]$ is equicontinuous since for each $k\in\nset^*$, $\hat R_k$ is Lipschitz continuous with  constant
$\lip{\hat R_k}=\sup_{\ell \in \nset^*} \vectornorm{r_\ell}$ .

Consider now $\sequence{\hat V}[k][\nset^*]$.  Since the sequence $\sequence{x}[k][\nset]$ belongs to the compact set $K$, under (A\ref{ass: limit_y_x})
the sequence $\sequence{y}[k][\nset]$ belongs to a compact neighborhood of $K$. By (A\ref{ass:open}) $F$ is locally bounded so
the sequence $\sequence{v}[k][\nset]$ is also bounded and  functions $\sequence{V}[k][\nset^*]$ are Lipschitz continuous with $\lip{\hat V_k}=\sup_{\ell \in \nset^*} \vectornorm{v_\ell}$, which implies equicontinuity.

Consider finally $\sequence{\hat E}[k][\nset^*]$.
For any arbitrarily chosen $t\in\rset_+$, $k,n \in \nset^*$, define
$b_{k,n}^t$
\[
b_{k,n}^t = \sup_{\{s : |s - t| <1/n \} } \vectornorm{\hat{E}_k(s) - \hat{E}_k(t)},\quad
\]
and consider the sequence $\sequence{a^t}[n][\nset]$ given for $n \in \nset^*$ by
\[
a_n^t = \sup_{k \geq 1} b_{k,n}^t\;.
\]
By construction $\sequence{a^t}[n][\nset^*]$ is nonincreasing and nonnegative. Hence $\sequence{a^t}[n][\nset^*]$
converges to some limit. Moreover, by assumption (A\ref{ass:noise}) the series $\sum_{k=1}^\infty \gamma_ke_k$
converges, so for any $\epsilon>0$, there exists $N\in\nset^*$ such that
\begin{equation}\label{eq:series}
\sup_{m\geq l\geq N}\vectornorm{\sum_{i=l}^m\gamma_i e_i} \leq\epsilon\;.
\end{equation}
First we observe that, since $\lim_{k \to \infty} s_k = \infty$,
the set $\mathcal{E}_N^t=\{k\colon s_k<t_N-t+1\}$ is
finite. Therefore, since for each $\ell \in \nset^*$ the functions $\hat{E}_\ell$ is continuous,  there exists $N'\geq N$ such that for any $k\in \mathcal{E}_N^t$ and
$n\geq N'$ we have $b_{k,n}^t\leq \epsilon$.  Now assume that
$k\not\in\mathcal{E}_N^t$.  Then for all $n\geq1$ and $s$ such that $|t-s|\leq
1/n$, we get $s + s_k \geq t_N$ and by \eqref{eq:series} for all $n\geq1$ and
$k\notin\mathcal{E}_N^t$ we get
\begin{equation*}
b_{k,n}^t= \sup_{\{s : |s - t| <1/n \} } \vectornorm{\hat{E}_k(s) - \hat{E}_k(t)}
\leq3\sup_{m\geq l\geq N}\vectornorm{\sum_{i=l}^m\gamma_i e_i}\leq 3\epsilon\;.
\end{equation*}
Therefore for $n\geq N'$ we have $a_n^t\leq 3 \epsilon$.
Since $\epsilon$ was arbitrary positive number, we get that $\lim_{n \rightarrow \infty} a_n^t = 0$. T
Hence, for all $\epsilon > 0$, there exists $N''  > 0$ such that
\[
\text{for all $s \in \rset$, $|t-s| < 1/N''$} \quad \Rightarrow \quad \sup_{k \geq
  1} \vectornorm{\hat{E}_k(t) - \hat{E}_k(s)} \leq \epsilon\eqsp.
\]
That is the family $\sequence{\hat{E}}[k][\nset^*]$  is pointwise equicontinuous at an arbitrary point $t$.
Together with equicontinuity of $\sequence{\hat V}[k][\nset^*]$ and $\sequence{\hat R}[k][\nset^*]$ it give us pointwise equicontinuity of $\sequence{X}[k][\nset^*]$.
Since by assumption $\sequence{x}[k][\nset^*]$ remains in the compact set $K$ so functions $\sequence{X}[k][\nset^*]$
are uniformly bounded  and we can apply Arzela-Ascoli theorem (\Cref{thm:Arzela-Ascoli}), showing that,  from every subsequence of $\sequence{X}[k][\nset^*]$, we can choose a
 further subsequence that converges uniformly on compact intervals, to some continuous limit $X_\infty$.
\end{proof}

\begin{step}
 \label{lem:limit_points}
 Any limit $X_\infty$ of converging subsequence is absolutely continuous and for almost every $t\geq 0$ there exists subsequence $\sequence{n}[k][\nset]$ such that
 \[
  X_\infty(t)=\lim_{k\to\infty} x_{n_k} =\lim_{k\to\infty} y_{n_k}\eqsp.
  \]
\end{step}
\begin{proof}

 Let $\subsequence{X}[n][k][\nset]$  be a subsequence that converges compactly to $X_\infty$.
 We start by proving that $X_\infty$ is absolutely continuous on compact intervals.
It is clear that $\subsequence{\hat E}[n][k][\nset]$ converges compactly to a function, that is equal to $\sum_{k=1}^\infty \gamma_k e_k =0$ everywhere.
That means, that the sequence
\begin{equation}
\label{eq:defintion-M-k}
M_k = x_0 + \hat{V}_{n_k} + \hat{R}_{n_k}
\end{equation}
converges compactly to the same limit as $X_{n_k}$.
Recall that $\hat V_k$ and $\hat R_k$  are Lipschitz continuous with constant independent on $k$.
Therefore all functions $M_k$ are Lipschitz continuous with common constant
and $X_\infty$, which is equal to  the limit of $\sequence{M}[k][\nset]$,  is also Lipschitz continuous and hence absolutely continuous on compact intervals.

For any $t \geq 0$ and $k \in \nset$ let us define $m(k,t)$ by
\begin{equation}\label{def:mkt}
m(k, t) = \min \{n \in \mathbb{N} : t_n > t + s_k \}\;,
\end{equation}
where $t_n$ is defined in \eqref{eq:def_t}. By assumption (A\ref{ass:stepsize}) $m(k,t)$ is well defined and converges to $\infty$ as $k\to\infty$.

By construction, for each $k \in \nset$ we get that
\begin{equation}
\label{eq:definition-X-n-k}
X_{n_k}(t_{m(n_k,t)}-s_{n_k})=x_{m(n_k,t)} \eqsp.
\end{equation}
Moreover, since $\lim_{k \rightarrow \infty} \gamma_k = 0$ and $\lim_{k \rightarrow \infty} s_k = \infty$, by \eqref{def:mkt} we have
\begin{equation}\label{eq:limeps}\lim_{k \rightarrow \infty} \{t_{m(n_k,t)}-s_{n_k}-t\} =0 \eqsp.\end{equation}
By the triangle inequality we get that
\begin{multline}\label{eq:y_minus_X}
\vectornorm {y_{m(n_k,t)} - X_\infty(t)} \\
\leq \vectornorm {y_{m(n_k,t)} - x_{m(n_k,t)}} + \vectornorm{ x_{m(n_k,t)} - X_\infty(t_{m(n_k,t)}-s_{n_k})} + \vectornorm{ X_\infty(t_{m(n_k,t)}-s_{n_k}) - X_\infty(t)}
\end{multline}
By assumption $s_k$ converges to $\infty$, so also $m(n_k,t)$ goes to $\infty$ as $k\to\infty$. Therefore by
assumption (A\ref{ass: limit_y_x}) the first part of the RHS of \eqref{eq:y_minus_X} converges to $0$.
By \eqref{eq:definition-X-n-k}, the second term in the RHS of \eqref{eq:y_minus_X} is equal to
\[\vectornorm{X_{n_k}(t_{m(n_k,t)}-s_{n_k}) - X_\infty(t_{m(n_k,t)}-s_{n_k})}\]
which goes to zero by uniform convergence of $X_{n_k}$ to $X_\infty$.
Finally, continuity of $X_\infty$ implies that the last term of \eqref{eq:y_minus_X} also converges to $0$.
All together we have therefore established that
\begin{equation}\label{eq:limy}
\lim_{k \rightarrow \infty} y_{m(n_k,t)} = X_\infty(t)\eqsp.
\end{equation}
\end{proof}

\begin{step}
 \label{lem:identify_weak_derivative}
  The limit $X_\infty$ is a solution of differential inclusion $\dot{x}(t) \in F(x(t))$.
\end{step}

\begin{proof}

We denote by  $G$ a weak derivative of $X_\infty$. We will prove that $G(t) \in F(X_\infty(t))$ for almost every $t \in\rset_+$.
By the definition \eqref{def:hat_functions}, for each $k \in \nset^*$ and almost every $t \in \rset_+$,  the weak derivatives of $M_{k}$
(see \eqref{eq:defintion-M-k}) at $t$ is equal to
$\dot{M}_k(t)=\hat v(t+s_{n_k})+\hat r(t+s_{n_k})$. Because $\sup_k (\vectornorm{ v_k}+\vectornorm{ r_k})<\infty$,
the functions $\sequence{\dot{M}}[k][\nset]$ are
uniformly integrable on finite intervals.
Thus, from \Cref{lem:weak_conv}, for any $0<a<b<\infty$
the sequence $\sequence{\indi{[a,b]}\dot{M}}[k][\nset]$  converges weakly to  $\indi{[a,b]}G$ in $L_1([a,b])$.
 From \Cref{lem:L1-conv-comb} there exists $\sequence{\dot{M}^w}[k][\nset]$ a convex combination subsequence
 (see \Cref{def:convex_combination}) of $\sequence{\dot{M}}[k][\nset]$ that converges to $G$ almost everywhere on $[a,b]$, \ie\ for almost every $t \in [a,b]$, $\lim_{k\to\infty}\dot{M}^w_k(t)=G(t)$.
By construction for any $t \in\rset_+$, we get $\dot M^w_k(t)=\hat v^w(t + s_{n_k})+\hat r^w(t + s_{n_k})$ and
\[
\hat v^w(t+s_{n_k})= \sum_{j=1}^\infty w_{n,j}   v_{m(n_j,t)} \quad \text{and} \quad \hat r^w(t+s_{n_k}) = \sum_{j=1}^\infty w_{n,j} r_{m(n_j,t)} \;.
\]
By assumption (A\ref{ass:noise}) for any $t\in\rset_+$, $\lim_{k\to\infty}\vectornorm{ r_{m(n_k,t)}}=0$ and hence $\lim_{k\to\infty}\hat r^w(t+s_{n_k}) =0 $
 It follows, that the for almost every $t\in[a,b]$ $\lim_{k\to\infty} \hat v^w(t+s_{n_k})=G(t)$. But we have for all $t\in\rset_+$, $v_{m(n_k,t)} \in F(y_{m(n_k,t)})$ and by \eqref{eq:limy} we know that $\lim_{k \to \infty} y_{m(n_k,t)}=X_\infty(t)$.
Since $F$ is upper hemicontinous and closed convex, we apply  \Cref{lem:convex_lim} to conclude that $G(t)\in F(X_\infty(t))$ for almost every $t\in[a,b]$.

We have proven that $G(t) \in F(X_\infty(t))$ for almost all $t \in [a,b]$, where $[a,b]\subset \rset_+$ is an arbitrary compact interval.
We can cover the real line $\rset_+$ by  a countable family of compact intervals of form $[0,\ell]$ for $\ell\in\nset^*$. Let  $\sequence{n}[k][\nset^*]$ be a sequence.
For each $\ell\in\nset^*$  can extract
subsequence $\sequence{n^{\ell}}[k][\nset^*]\subseteq\sequence{n^{\ell-1}}[k][\nset^*]$ such that $\{X_{n^\ell_k},\eqsp k\in\nset^*\}$ converges uniformly on $[0,\ell]$.
Setting $\tilde n_k=n_k^k$ we get that there exists function $X_\infty: \rset\mapsto\rset^d$ such that that for any
    $T>0$,
    $\lim_{k \to \infty}\sup_{t \in \ccint{0,T}} \vectornorm{X_{\tilde
        n_k}(t) - X_\infty(t)} = 0$.  and   $\dot X_\infty(t) \in F(X_\infty(t))$ for almost all $t \in \rset$.
This is equivalent to saying, that $X_\infty$ is a solution of the differential inclusion $\dot{x}\in F(x)$.
\end{proof}

\end{proof1}

Combining \Cref{thm:limit_compact}  with stability properties of underlying differential inclusion $\dot x\in F(x)$ we establish convergence of PADs.
To state the result we need to define a set valued Lie derivative, introduced in \cite{bacciotti1999stability}.
\begin{definition}\label{def:Lie-derivative}
Let $\Xset \subseteq \rset^d$  be an open subset, $F: \Xset \rightarrow \rset^{n}$ be a set-valued map, and $V:\rset^d\to \rset_+$ be a locally Lipschitz function.
The set-valued Lie derivative of $V$ with respect to $F$ at $x$  is defined by
\[
\mathcal{L}_{F}V (x) = \left\{a \in \rset : \exists { v \in F(x)}\text{ such that } \pscal{v} {w} = a\eqsp, \forall w \in \bar{\partial}V(x) \right\}\eqsp,
\]
where $\bar{\partial}V$ is the Clarke generalized gradient of $V$, cf. \Cref{def:clarke-gradient}.
\end{definition}
The Lie derivative plays important role in analysis of stability of solution of differential inclusions. We in particular will used an important property stated in \cite[Lemma 1]{bacciotti1999stability}.
\begin{lemma} \label{lem:Lie-Lyapunov-regular}
Let $F: \Xset \rightarrow \mathcal{P}(\rset^d)$ be a set-valued map on an open domain $\Xset$, $I\in\rset$ be an interval, and assume that there exists  $\phi: I \rightarrow \Xset$ a solution of the differential inclusion $\dot{x} \in F(x)$. Let $V$ be a locally Lipschitz regular function defined on $\Xset$.
Then $\frac{d}{dt} V(\phi(t))$ exists for almost all $t \in I$, and  for almost all $t \in I$ we have:
\[
\frac{d}{dt} V(\phi(t)) \in \mathcal{L}_F V(\phi(t))
\]
where $\mathcal{L}_F V$ is the set-valued Lie derivative of $V$ with respect to $F$, cf. \Cref{def:Lie-derivative}.
\end{lemma}

\begin{theorem}\label{thm:convergence}
  Let $\Xset\subseteq \rset^d$  be an open subset and  $F: \Xset\mapsto\mathcal{P}(\rset^d)$ be a set valued map,
$\sequence{x}[k][\nset^*]$ be a PAD of $F$ with step sizes
  $\sequence{\gamma}[k][\nset^*]$ and perturbations
  $\sequence{\eta}[k][\nset^*]$. Assume that conditions
  (A\ref{ass:open}--\ref{ass: limit_y_x}) hold, and there exists a
  compact set $K \subset \Xset$ such that $x_k \in K$ for any
  $k \geq 1$.
Let $V : \Xset \rightarrow \rset$ be a locally Lipschitz, regular
function.  Suppose that there exists an upper semicontinuous function
$U : \Xset \rightarrow \rset$, such that for all $x\in K$
\[
\sup \mathcal{L}_F V(x) \leq U(x) \leq 0,
\]
and set $\mathcal{S} \eqdef \{x \in \Xset : U(x) = 0 \}$.
\begin{enumerate}[(i)]
\item The image by $V$ of the set of limiting points of
  $\sequence{x}[k][\nset^*]$ is a compact interval in
  $V(\mathcal{S} \cap K)$.
\item If $V(\mathcal{S} \cap K)$ has empty interior, then $\sequence{x}[k][\nset^*]$
converges to $K \cap \mathcal{S}$.
\end{enumerate}
\end{theorem}
\begin{proof}
\setcounter{St}{0}
The proof is divided into four steps.

\begin{step}
 The set $K \cap \mathcal{S}$ is nonempty.
\end{step}
\begin{proof}
By \Cref{thm:limit_compact}, there exists a solution  $X:\rset_+\mapsto \rset^d$ of the differential inclusion $\dot{x}\in F(x)$ satisfying $X(t)\in K$ for all $t\in\rset_+$.
By \Cref{lem:Lie-Lyapunov-regular} for almost all $t\in\rset_+$ $\frac{d}{dt} V(X(t))$
is well-defined and  we have:
\begin{equation*}
\frac{d}{dt} V(X(t)) \in \mathcal{L}_FV(X(t))\eqsp.
 \end{equation*}

If the $K \cap \mathcal{S}=\emptyset$  by upper semicontinuity of $U(x)$ and compactness of $K$
 we would have $\sup_{x \in K} U(x) = -\delta$ for some $\delta > 0$.
Therefore function $V \circ X$ must decrease at a rate at least $\delta$, and thus $\lim_{t \rightarrow \infty} V(X(t)) = -\infty$.
But this is a contradiction with the assumption that $V$ is bounded from below.
\end{proof}
Since $V$ is continuous and $K$ is compact,
\begin{equation}
\label{eq:definition-liminf}
L= \liminf_{k \in \nset} V(x_k) > - \infty \eqsp,
\end{equation}
and there exists $x_* \in K$ and a subsequence $\subsequence{x}[n][k][\nset^*]$ such that $\lim_{k\to\infty}x_{n_k}=x_*$  and $V(x_*)= L$.
\begin{step}\label{step:liminf_x}
If   $\subsequence{x}[n][k][\nset^*]$ is a subsequence  such  that $\lim_{k\to\infty}x_{n_k}=x_*$ and  $V(x_*)=L$ then $x^*\in\mathcal{S} \cap K$.
\end{step}

\begin{proof}
The proof is by contradiction. Assume that  
$x_* \not\in \mathcal{S} \cap K$. Therefore, we can find disjoint
open neighborhoods 
 $\Xset \supset A \supset \mathcal{S} \cap K$ and by $\Xset \supset B \ni x_*$.
Also, there exists $r > 0$, such that $\closure{\ball}(x_*, r)\eqdef\{y\in\rset^d\colon \vectornorm{y-x^*} \leq r\} \subseteq B$.
Define
\begin{equation}\label{eq:vmax}
v_{\max} = \sup_{x \in K} \sup_{z \in F(x)} \vectornorm{z},
\end{equation}
which is finite by local boundedness of $F$. 
We denote by $\Delta t = r/v_{\max}$. From Theorem~\ref{thm:limit_compact}  it follows, that on the interval  $[0,\Delta t]$
there exists a subsequence $\sequence{\tilde n}[k][\nset^*]\subseteq \sequence{n}[k][\nset^*]$ such that
$\lim_{k\to\infty}\sup_{t\in[0,\Delta t]} \vectornorm {X_{\tilde{n}_k}-X_\infty} =0$, where $X_\infty$ is a solution of $\dot x\in F(x)$.
Let $X_0$ be defined by \eqref{eq:def_X} and set, for all $k\in\nset$, $s_k=t_{\tilde{n}_k}$, where $\sequence{t}[k][\nset]$ is defined in \eqref{eq:def_t}.
By (A\ref{ass:stepsize}) $\lim_{k\to\infty}{s_k}=\infty$.
Consider $\sequence{X}[k][\nset]$ defined by \eqref{eq:shift}.
By definition of $\sequence{s}[k][\nset]$, for all $k$, we have $X_{k}(0) = x_{\tilde n_k}$ and since $\lim_{k\to\infty}x_{\tilde n_k}=x_*$
 we have $X_\infty(0) = x_{*}$. For almost every $t \in [0,\Delta t]$ we have $\dot X_\infty(t) \in F(X_\infty(t))$
 and hence by \eqref{eq:vmax} we get $\vectornorm {\dot X_\infty(t)} \leq v_{\max}$. Thus, for all $t\in[0,\Delta t]$,
\[
\vectornorm {X_\infty(t) - X_\infty(0)} = \vectornorm{ \int_{0}^t \dot X_\infty(s) \rmd s } \leq \int_{0}^t v_{\max} \rmd s \leq v_{\max} \Delta t  = r\eqsp,
\]
and hence $X_\infty(t) \in \ball(x_*, r)$. By upper semicontinuity of $U$ and compactness of $\closure{\ball}(x_*,r)$ we
get that there exists $\delta>0$ such that $\sup_{x \in \closure{B}(x_*,r)} U(x) = - \delta$, and,
using \Cref{lem:Lie-Lyapunov-regular}, we conclude that for almost every $t \in [0, \Delta t]$, $\frac{d}{dt} V(X_\infty(t)) \leq -\delta$.
This means, that
\begin{equation}\label{eq:vdelta}
V(X_\infty(\Delta t)) \leq V(X_\infty(0)) - \delta \Delta t = L - \delta \Delta t\;,
\end{equation}
where $L$ is defined in \eqref{eq:definition-liminf}. But, by \Cref{thm:limit_compact} for almost every $t\in[0,\Delta t]$, $X_\infty(t)$ is a
accumulation point of sequence $\sequence{x}[k][\nset^*]$ so \eqref{eq:vdelta} contradicts with $\liminf_{k \rightarrow \infty} V(x_k) = L$.
\end{proof}
\begin{step}\label{step:KS}
 If $\hat x$  is an accumulation point of the sequence $\sequence{x}[k][\nset^*]$, then  $V(\hat x)\in V(\mathcal{S} \cap K)$.
\end{step}
\begin{proof}
The proof is by contradiction. Assume that there exists a subsequence  $\subsequence{x}[n][k][\nset^*]$ such that  $\lim_{k\to\infty} x_{n_k}=\hat{x}$ and $V(\hat{x}) \not\in V(\mathcal{S} \cap K)$.
Then we have $\hat{x} \not\in \mathcal{S} \cap K$ as well. By \Cref{step:liminf_x}, we know that there exists another subsequence $\sequence{m}[l][\nset^*]$
such that $\lim_{l\to\infty} x_{m_l}={x_*}$, $V(x_*)=\liminf_k V(x_k)=L$
and $x_*\in \mathcal{S} \cap K$, where $L$ is defined in \eqref{eq:definition-liminf}. Note that $V(x_*)<V(\hat x)$. So, there exist $a_1<a_2<b_1<b_2$ such that intervals $\ooint{a_1,a_2}\ni V(x_*)$ , $\ooint{b_1,b_2}\ni V(\hat x)$ are disjoint and
 $\ooint{a_2,b_2}\cap V(\mathcal{S} \cap K)=\emptyset$.
We denote by   $A = V^{-1}(\ooint{a_1,a_2})$ and by  $B = V^{-1}(\ooint{b_1,b_2})$. Observe that sets $A$ and $B$ are also disjoint.

Since $\hat{x}\in B$ and $x_*\in A$, the function $X_0$, defined in \eqref{eq:def_X}, must go from $A$ to $B$ infinitely often.
More precisely, for $j\in\nset^*$ we can define three increasing sequences $\sequence{l}[j][\nset]$, $\sequence{\tilde l}[j][\nset]$, and  $ \sequence{r}[j][\nset]$ by recurrence as follows:
$r_0=l_0=\tilde l_0 =0$ and for $j\geq 1$, $\tilde l_j=\min\{t\geq r_{j-1}\colon X_0(t)\in A\}$, $r_j=\min\{t\geq \tilde l_{j}\colon X_0(t)\in B\}$ and
$l_j=\max\{t\leq r_j\colon X_0(t)\in A\}$.  Because $X_0$ is continuous, the sequences
$\sequence{l}[j][\nset],\eqsp \sequence{r}[j][\nset]$ are well defined. Since, sets $A$ and $B$ are disjoint and both contain accumulation points of $\sequence{x}[k][\nset^*]$,
by construction $\lim_{j \rightarrow \infty} l_j = \infty$, in
addition by continuity of $V$ we get  that $V(X_0(l_j)) = a_2$,
$V(X_0(r_j)) = b_1$.  Furthermore, for all $t \in (l_j, r_j)$ we have $V(X_0(t)) \in (a_1,b_2)$ and $X_0(t) \in K \setminus (A \cup B)$.

Consider the  sequence $\{r_j-l_j:\eqsp j\in\nset^*\}$ of positive numbers.  Set $S= \limsup_{j \to \infty} \{r_j - l_j\} \in \ccint{0,\infty}$ and let $\sequence{m}[j][\nset]$ be a sequence such that $\lim_{j\to\infty} \{r_{m_j} - l_{m_j}\}= S$.

Let $s_k=l_{m_k}$, $T>0$ and consider $\sequence{X}[k][\nset^*]$ defined in \eqref{eq:shift}.
By \Cref{thm:limit_compact} there exists $\sequence{\tilde n}[k][\nset^*]$ such that
$\lim_{k\to\infty}\sup_{t\in[0,T]} \vectornorm {X_{\tilde{n}_k}-X_\infty} =0$, where $X_\infty$ is a solution of $\dot x\in F(x)$.
First assume that $S=0$. By construction for all $k\in \nset^*$, $V(X_k(0))=a_2$ and $V(X_k(r_{m_k}-l_{m_k}))=b_1$. But this contradicts the continuity of $V\circ X_\infty$. Therefore we must have $S>0$.

Consider now the case: $S>0$. By construction, we may find $\tilde{S} \in \ooint{0,S}$ such that for large enough $k$ for all $t\in \ocint{0,\tilde{S}}$ we have $X_k(t)\in K \setminus (A \cup B)$ and $V(X_k(t))>a_2$. Therefore, for all $t\in \ocint{0,\tilde{S}}$ the limit $X_\infty$ also satisfies
$X_\infty(t)\in K \setminus (A \cup B)$, $V(X_\infty(t))>a_2$, $V(X_\infty(0))=a_2$. On the other hand, by upper semicontinuity of $U$ and compactness of $K \setminus (A \cup B)$ we
get that there exists $\delta>0$ such that $\sup_{x \in K \setminus A} U(x) = - \delta$ and by \Cref{lem:Lie-Lyapunov-regular} $V\circ X_\infty$
is strictly decreasing on $[0, \tilde{S}]$, which contradicts with $V(X_\infty(t))>a_2$, $V(X_\infty(0))=a_2$. Hence $V(\hat x)\in V(\mathcal{S} \cap K)$.

\end{proof}
\begin{step}\label{step:m}
If $\hat x$ is an accumulation point of $\sequence{x}[k][\nset]$, then $V(\hat x) =L$, where $L$ is defined in \eqref{eq:definition-liminf}.
\end{step}
\begin{proof}
Consider $X_0$, defined in \eqref{eq:def_X}.
Suppose there were two different points $x', x''$, that are accumulation points of $\sequence{x}[k][\nset^*]$,
with $V(x') < V(x'')$.  Then the sequence $\{ V(x_k), k \in \nset^*\}$ must oscillate between
between $V(x')$ and $V(x'')$. Let $v \in \ccint{V(x'),V(x'')}$.
We  define the sequence $\sequence{s}[k][\nset]$ by $s_0=0$ and for $k\geq 1$, $s_k=\min\{t> s_{k-1} \colon V(X_0(t))=v\}$, which is well defined
by continuity of $V\circ X_0$. Let $X_k$ be defined by \eqref{eq:shift}. For any $T>0$ by \Cref{thm:limit_compact} there exists $\sequence{\tilde n}[k][\nset^*]$ such that
$\lim_{k\to\infty}\sup_{t\in[0,T]} \vectornorm {X_{\tilde{n}_k}(t)-X_\infty(t)} =0$, where by construction $V(X_\infty(0))=v$ and $x'''= X_\infty(0)$ is accumulation point of $\sequence{x}[k][\nset^*]$.

Using \Cref{step:KS}, $V(x'),V(x''),V(x''')\in V(\mathcal{S} \cap K)$. Since $v\in[V(x'),V(x'')]$ is arbitrary, therefore, the whole interval $[V(x'), V(x'')]$
must be contained in $V(\mathcal{S} \cap K)$.
But this contradicts our assumption, that $V(\mathcal{S} \cap K)$ has empty interior.
Therefore, for any point $\hat{x}$ that is an
accumulation point of the sequence $\sequence{x}[k][\nset]$ we must have $V(\hat{x}) = L$.
\end{proof}
We can now conclude the proof of \Cref{thm:convergence}. Combining \Cref{step:liminf_x} with \Cref{step:m} we get that  any accumulation point of $\sequence{x}[k][\nset^*]$ belongs to $\mathcal{S} \cap K$ , and hence $\lim_{k \rightarrow \infty} d(x_k,  \mathcal{S}\cap K) = 0$.
\end{proof}

\section{Convergence of Projected Perturbed Approximate Discretisation}
\label{sec:convergencePPAD}
Let $K\in\rset^d$ be a compact convex set .
For any $x\in K$ the tangent and the normal cone are set valued maps defined as
\begin{align}
\label{eq:tangent_cone}
 \tangentcone{K}(x)&=\mathrm{cl}\{ d\in\rset^d\colon \text{$\exists$ $\epsilon>0$ such that $x+\epsilon d\in K$}\}\eqsp, \\
\label{eq:normal_cone}
 {\normalcone{K}}(x)&=\{ g\in\rset^d\colon \pscal{g}{z-x}\leq0,\text{ for all $z\in K$}\}\eqsp,
\end{align}
where $\mathrm{cl}(A)$ denotes closure of set $A$. Let $F$ be a convex compact set valued map defined on $\Xset\subseteq\rset^d$.  Consider the $K$-PPAD sequence $\sequence{x}[k][\nset]$ defined in \eqref{eq:ppad}: for $k \in\nset^*$,
\begin{equation}
\label{eq:PPAD-1}
x_k = \Pi_K \left( x_{k-1} + \gamma_k \{ v_k + e_k+r_k \} \right) \;,
\end{equation}
where $\sequence{e}[k][\nset]$ and $\sequence{r}[k][\nset]$ are defined in A\ref{ass:noise} and  $v_k \in F(y_k)$.

\begin{theorem}\label{thm:limit_projected}
Let $\sequence{x}[k][\nset^*]$ be the $K$-PPAD given in \eqref{eq:PPAD-1}.
Assume that conditions (A\ref{ass:open}--\ref{ass: limit_y_x}) hold and $\sup_k\vectornorm {e_k}<\infty$.
 Then
 \begin{enumerate}[(i)]
  \item The sequence of functions $\sequence{X}[k][\nset^*]$, defined in \eqref{eq:shift}, is precompact in the topology of compact convergence.
  \item  For every convergent subsequence of $\sequence{X}[k][\nset^*]$ the limit $X_\infty$ is a solution of a projected differential inclusion $\dot{x} \in \Pi_{{\tangentcone{K}}(x)}(F(x))$, where $\Pi_{{\tangentcone{K}}(x)}$ is the projection onto the closed convex cone ${\tangentcone{K}}(x)$.
 \end{enumerate}

\end{theorem}
\begin{proof1}
As in the proof of \Cref{thm:limit_compact}, we  assume $\sum_{k=1}^{\infty} \gamma_k e_k = 0$.
We denote by
\begin{equation}\label{def:p}
p_k = \frac{x_k - x_{k-1}}{\gamma_k} - v_k - e_k - r_k\;,\end{equation}
and we define the functions for any $t\geq 0$
\begin{equation*}
\hat p (t) =\sum_{k=1}^\infty p_k \indi{\coint{t_{k-1},t_k}}(t)
\;,\quad \hat P_0(t)=\int_0^t\hat p(s) \rmd s\eqsp,
\end{equation*}
where $t_k$ is defined in \eqref{eq:def_t}.
Let $\sequence{s}[k][\nset^*]$ be an increasing sequence of positive real numbers, such that $\lim_{k \to \infty} s_k= \infty$.
For any $k\in\nset^*$ and $t\in\rset_+$ we define
\begin{equation}
\hat P_k(t) = \hat P_0(t + s_k) - \hat P_0(s_k)\eqsp.
\label{eq:def_P_k}
\end{equation}
\setcounter{St}{0}
The proof is divided into three steps.
\begin{step}
 The family of functions $\{(X_k, \hat P_k)\colon k\in\nset\}$, defined by
    \eqref{eq:shift} and \eqref{eq:def_P_k}, respectively,  is precompact in the topology of compact
    convergence.
\end{step}
\begin{proof}
The boundedness of $\sequence{X}[k][\nset]$ is trivial due to the fact, that $x_k \in K$ for all $k \geq 0$.
Using \eqref{def:p} and noting that, since projection is a contraction,
\begin{equation}
 \label{eq:bounded_x}
 \gamma_k^{-1}\vectornorm{x_k-x_{k-1}}\leq \vectornorm{v_k}+\vectornorm{r_k}+\vectornorm{e_k}
\end{equation}
we get that
$\sup_k\vectornorm{p_k}<\infty$.  Next, since
\[
 \hat P_k(t)=\int_{s_k}^{t+s_k}\hat p(s) \rmd s\eqsp,
\]
and for every $s\geq 0$, $\vectornorm{\hat p(s)}\leq \sup_k\vectornorm{p_k}$ we get $\vectornorm{\hat P_k(t)}\leq \sup_k\vectornorm{p_k} t$.

In addition since $\sup_k\vectornorm{p_k}<\infty$ and by \eqref{eq:bounded_x} $\sup_k\gamma_k^{-1}\vectornorm{x_k-x_{k-1}}<\infty$, for all $k\in\nset$ the functions
$\hat P_k$, and $X_k$  are Lipschitz continuous with a Lipschitz constant, which does not depend on $k$. Hence $\{(X_k,P_k),\eqsp k\in\nset^*\}$ is equicontinuous.

By Arzela-Ascoli theorem thee exists subsequence $\sequence{n}[k][\nset]$ such that $\subsequence{X}[n][k][\nset]$ converges
uniformly on compact subsets to a limit $X_\infty$ and $\subsequence{\hat P}[n][k][\nset]$ converges uniformly on compact subsets
to $P_\infty$.
\end{proof}
\begin{step}\label{step:step2}
 Any limit $(X_\infty, P_\infty)$ of converging subsequence is Lipschitz continuous and for almost every $t\geq 0$ there exists subsequence $\sequence{n}[k][\nset]$ such that
 \[
  X_\infty(t)=\lim_{k\to\infty} x_{n_k} =\lim_{k\to\infty} y_{n_k}\eqsp.
  \]
\end{step}
\begin{proof}
Since $\sequence{X}[k][\nset^*]$ and $\sequence{\hat P}[k][\nset^*]$ are Lipschitz continuous, so  the limits $X_\infty$, $P_\infty$ are also
Lipschitz continuous. Along the same lines as in \Cref{lem:limit_points} of proof of \Cref{thm:limit_compact}, i.e. from \eqref{eq:definition-X-n-k}, \eqref{eq:limeps} and \eqref{eq:y_minus_X},
we get that
\begin{equation}\label{eq:limy_proj}
\lim_{k \rightarrow \infty} y_{m(n_k,t)} = X_\infty(t)\eqsp,
\end{equation}
where $m(n_k,t)$ is defined in \eqref{def:mkt}.
\begin{step}
 Any limit $X_\infty$ of converging subsequence is a solution of the projected differential inclusion
   $\dot{x} \in \Pi_{{\tangentcone{K}}(x)}(F(x))$.
\end{step}
We use the following characterization of the solution of projected differential inclusion $\dot{x} \in \Pi_{{\tangentcone{K}}(x)}(F(x))$ given in \cite[Chapter 5, Section 6, Propositions 1 and 2]{aubin2012differential}:

\begin{enumerate}[(i)]
\item\label{item2} $X_\infty$ is absolutely continuous
\item\label{item1} For all $t\geq 0$ we have $X_\infty(t) \in K$.
\item\label{item3} For almost every $t\geq 0$ there exists $w(X_\infty(t)) \in F(X_\infty(t))$ such that,
\[ \dot X_\infty(t) - w(X_\infty(t))\in -{\normalcone{K}}(X_\infty(t)).\]
\end{enumerate}
The  condition \eqref{item2} is already proved in Step~\ref{step:step2} .
First we show \eqref{item1}.
Since $K$ is a convex and by construction for all $k\in\nset$ and $t\geq0$,  $X_k(t)$ is convex combination of two elements of $K$,
$X_k(t)\in K$. Because $K$ is compact and $X_\infty$ is the limit of convergent subsequence $\subsequence{X}[n][k][\nset]$ also
for all $t\geq 0$,
$X_\infty(t)\in K$.

It remains to show the condition \eqref{item3}. We choose
a compact interval $[a,b] \subset \rset_+$. Let $G$ and $Q$ denote  the weak derivatives of $X_\infty$
and $P_\infty$, respectively.
Recall that by assumption (A\ref{ass:noise}) the family of functions $\sequence{\hat{E}}[k][\nset]$ is uniformly bounded on $\rset$, and converges uniformly to $0$ on compact intervals.
Hence $\{(X_{n_k}-\hat E_{n_k}, \hat P_{n_k}),\eqsp k\in\nset\}$ converges uniformly on compact sets to the limit
$(X_\infty, P_\infty)$. Because the functions $\{(X_{n_k} - \hat E_{n_k}, \hat P_{n_k}),\eqsp k\in\nset\}$ are Lipschitz
continuous with the same constant, their weak derivatives $\{(G_{n_k},  Q_{n_k}),\eqsp k\in\nset\}$ are uniformly integrable. Hence, applying \Cref{lem:weak_conv} we get that
$\{(G_{n_k},  Q_{n_k}),\eqsp k\in\nset\}$ converges in the weak topology of $L_1([a,b])$ to $(G,Q)$.
By \Cref{lem:L1-conv-comb}, we conclude that there exists a convex combination subsequence
$\{(G^w_{n_k},  Q^w_{n_k}),\eqsp k\in\nset\}$ that
converges to $(G, Q)$ for almost every $t\in[a,b]$.

Let $t \in [a,b]$ be such, that $\lim_{k\to\infty}(G^w_{n_k}(t),  Q^w_{n_k}(t))=(G(t),  Q(t))$.
Since $Q_{n_k}(t)=p_{m(n_k,t)}$, where $m(n_k,t)$ is defined in \eqref{def:mkt}, $Q(t)=\lim_{k\to\infty}p^w_{m(n_k,t)}$.

Inequality \eqref{def:proj_alternative} shows that for any $l\in\nset$ and  $z\in K$,
$
\langle x_l - z, p_l \rangle \leq 0
$, or equivalently $-p_l\in {\normalcone{K}}(x_l)$, see definition \eqref{eq:normal_cone}.
By \cite[Chapter 5, Section 1, Theorem 1]{aubin2012differential} the normal cone has closed graph.
On the other hand, since $\sup_l\vectornorm{p_l}=r<\infty$ ,
by \cite[Chapter 1, Section 1, Theorem 1]{aubin2012differential} the map
$x\mapsto {\normalcone{K}}(x)\cap \ball(0,r)$ is upper hemicontinuous.
Therefore, since $\lim_{k\to\infty}x_{m(n_k,t)}=X_\infty(t)$, we can apply \Cref{lem:convex_lim} to
conclude that $Q(t)=\lim_{k\to\infty}Q^w_{n_k}(t)$ belongs to the minus normal
cone of $K$ at $X_\infty(t)$, \ie $-Q(t) \in {\normalcone{K}}(X_\infty(t))$.

On the other hand, $\lim_{k\to\infty}(G^w_{n_k}(t) - Q^w_{n_k}(t))=G(t)-Q(t)$ and since for any $m\in \nset$
\[
 \frac{x_{m+1} - x_{m}}{\gamma_{m+1}} - \left(\frac{x_{m+1} - x_{m}}{\gamma_{m+1}}
- v_{m} - e_{m} - r_{m} \right) - e_{m} = v_{m} + r_{m}\;
\]
we get $G^w_{n_k}(t) - Q^w_{n_k}(t) = v^w_{m(n_k,t)} + r^w_{m(n_k,t)}$.

For all $t\geq 0$, since $\lim_{k\to\infty}\vectornorm{r_{m(n_k,t)}}=0$, we have $\lim_{k\to\infty}\vectornorm{r^w_{m(n_k,t)}}=0$.
Therefore, for almost every $t\in[a,b]$, we get that $\lim_{k\to\infty}v^w_{m(n_k,t)}= G(t)-Q(t)$. On the other hand, for all
$t\in\rset_+$, $v_{m_(n_k,t)}\in F(y_{m(n_k,t)})$ and by \eqref{eq:limy_proj}, $\lim_{k\to\infty}y_{m(n_k,t)}=X_\infty(t)$.
Since $F$ is upper hemicontinuous closed convex we apply \Cref{lem:convex_lim} to show that $G(t) - Q(t)\in F(X_\infty(t))$.
We denote $w(X_\infty(t))=G(t)-Q(t)$. Hence for almost every $t\in[a,b]$ we have:
\begin{equation}
\label{eq:-nc}
\dot X_\infty(t) - w(X_\infty(t))= G(t)-(G(t)-Q(t))=Q(t) \in  - {\normalcone{K}}(X_\infty(t)).
 \end{equation}

 We therefore proved, that for almost all $t \in [a,b]\subset\rset_+$  there exists $w(\dot X_\infty(t)) \in F(X_\infty(t))$, such that
 the \eqref{eq:-nc} holds. Since there is countable cover of real line by compact intervals, by diagonal extraction argument we can define $w(\dot X_\infty(t))$, which
 satisfies \eqref{eq:-nc} for almost all $t \in \rset_+$.  Hence condition \eqref{item3} holds, and that concludes the proof.
$\ $
\end{proof}
\end{proof1}

\begin{theorem}\label{thm:convergence_proj}
   Let $K\subset\rset^d$ be a convex and compact, $\Xset\subseteq \rset^d$  be an open subset and  $F: \Xset\mapsto\mathcal{P}(\rset^d)$ be a set valued map,
$\sequence{x}[k][\nset^*]$ be a $K$-PPAD of $F$ with step sizes
  $\sequence{\gamma}[k][\nset^*]$ and perturbations
  $\sequence{\eta}[k][\nset^*]$. Assume that conditions
  (A\ref{ass:open}--\ref{ass: limit_y_x}) hold and $\sup_k\vectornorm {e_k}<\infty$..
Let $V : \Xset \rightarrow \rset$ be a locally Lipschitz, regular
function.  Suppose that there exists an upper semicontinuous function
$U : \Xset \rightarrow \rset$, such that for all $x\in K$
\[
\sup \mathcal{L}_{\Pi_{{\tangentcone{K}}}(F)}V(x) \leq U(x) \leq 0,
\]
and set $\mathcal{S} \eqdef \{x \in \Xset : U(x) = 0 \}$.
\begin{enumerate}[(i)]
\item The image by $V$ of the set of limiting points of
  $\sequence{x}[k][\nset^*]$ is a compact interval in
  $V(\mathcal{S} \cap K)$.
\item If $V(\mathcal{S} \cap K)$ has empty interior, then $\sequence{x}[k][\nset^*]$
converges to $K \cap \mathcal{S}$.
\end{enumerate}
\end{theorem}
\begin{proof}
 The proof is along the same lines as the proof of \Cref{thm:convergence}.
\end{proof}

\section{Applications}
\label{sec:applications}
In this section we apply the result from previous sections to projected \proxSGD and projected subgradient descent algorithm.
\subsection{Proximal stochastic gradient descent algorithm}
Stochastic proximal gradient is a natural extension of Proximal Gradient algorithm to the case where the gradient cannot be computed exactly and is therefore  affected by some errors. More specifically, we want to optimize a composite function of form
\[P(x) = f(x) + g(x)\;,\] where $f$ is a continuously differentiable function on some open set $\Xset \subseteq \rset^d$, and $g$ is locally Lipschitz,
bounded from below, regular function (see \Cref{def:regular-function}).

In many Machine Learning applications $f$ is  the empirical risk  of the model, and $g$ a sparsity inducing penalties like LASSO \cite{Tibshirani2011Regression}, MCP \cite{Zhang2010Nearly}, or SCAD  \cite{fan2001Variable}.
We consider applications in which the gradient  $\nabla f(x_k)$ cannot be computed but for which  noisy estimates $H_k$ of $\nabla f(x_k)$ are available. It is well known that, even in the noiseless case, additional  conditions are needed to guarantee that the successive iterates remains in compact set and further converge to stationary points. To overcome this issue we consider a projected version of algorithm. For a predefined compact set $K \subseteq \Xset$, we choose a sequence of step sizes $\sequence{\gamma}[k][\nset]$ satisfying (A\ref{ass:stepsize}) and a starting point $x_0 \in K$.

Denote by  $\cvxind_K$  the convex indicator function of set $K$ ($\cvxind(x)= 0$ if $x \in K$ and $\cvxind(x) =\infty$ otherwise) and by $\prox$ the proximal operator  (see \eqref{eq:proximal}).

We consider two versions of the projected \proxSGD\ algorithm which are given by
\begin{equation}\label{stochProxGrad_2}
x_k \in \prox_{\gamma_k (g + \cvxind_K)} (x_{k-1} - \gamma_k H_k)
\end{equation}
or
\begin{equation}\label{stochProxGrad_1}
x_k \in \Pi_K\left( \prox_{\gamma_k g} (x_{k-1} - \gamma_k H_k) \right)
\end{equation}
Those two approaches to projection are not equivalent, and depending on $g$ and $K$ one might be easier to compute than the other. Consider the following assumptions:
\begin{hypP}
\label{assP:field}
$\Xset \subset \rset^d$ be open set, $f: \Xset \to \rset$ be a continuously differentiable function, and $g: \Xset \to \rset $ be a
locally Lipschitz, bounded from below, regular function (see \Cref{def:regular-function}).
\end{hypP}

\begin{hypP}
\label{assP:stepsize}
The sequence of step sizes $\sequence{\gamma}[k][\nset^*]$ satisfies $\gamma_k > 0$, $\sum_{k=0}^{\infty} \gamma_k = \infty$, and $\lim_{k \rightarrow \infty} \gamma_k = 0$.
\end{hypP}
Denote by $\sequence{\delta}[k][\nset^*]$  the gradient perturbation
\begin{equation}
\label{eq:definition-delta}
\delta_k = H_{k} - \nabla f(x_{k-1}) \eqsp.
\end{equation}
\begin{hypP}
\label{assP:noise}
The sequence $\sequence{\delta}[k][\nset^*]$ can be decomposed as $\delta_k= e^\delta_k + r^\delta_k$ where $\sequence{e^\delta}[k][\nset^*]$ and $\sequence{r^\delta}[k][\nset^*]$ are two sequences satisfying $\lim_{k \to \infty} \vectornorm{r^\delta_k} = 0$ and the series $\sum_{k=1}^\infty \gamma_k e^\delta_k$ converges.
\end{hypP}
To illustrate our derivations, we consider now two possible choices of sparsity inducing penalties $g$.

\begin{example}[MCP penalty]
THe MCP penalty, introduced in \cite{Zhang2010Nearly} is given  for $x = (x_1,\dots,x_d)\in \rset^d$ by $g(x)= \sum_{i=1}^d p_{\lambda,\kappa}(x_i)$ where for $z \in \rset$, $\lambda > 0$ and $\kappa > 1$,
\[
p_{\lambda,\kappa}(z) =
\begin{cases}
\lambda |z| - z^2/ (2 \kappa) & \text{if $|z| \leq \kappa \lambda$} \\
1/(2 \kappa \lambda^2) & \text{otherwise}.
\end{cases}
\]
The function $g$ is Lipschitz on $\rset^d$ and nonnegative.
If $z \not= 0$, $p_{\lambda,\kappa}$ is differentiable at $z$ and hence $p_{\lambda,\kappa}$ is regular at $z$ by \cite[Proposition~2.3.6]{clarke1990optimization}. The function $z \mapsto \lambda |z|$ is convex and therefore regular on $\rset$, applying again \cite[Proposition~2.3.6]{clarke1990optimization}. The function $z \mapsto -z^2/(2 \kappa)$ is differentiable on $\rset$ and therefore regular. The sum of regular functions being regular by \cite[Proposition~2.3.6]{clarke1990optimization}, $z \mapsto \lambda |z| - z^2/ (2 \kappa)$ is regular on $\rset$. Therefore, $p_{\lambda,\kappa}$ is regular at $0$.
By \Cref{lem:regularity-sum}, the function $g$ is also regular and satisfy (P\ref{assP:field}).
The proximal operator for the MCP penalty $p_{\lambda,\kappa}$ is the function given (see for example \cite{Breheny2011Coordinate}), for all $\gamma \in \ooint{0,\kappa}$ by
\begin{equation*}
\prox_{\gamma p_{\lambda,\kappa}}(z)=
\begin{cases}
S(z,\gamma \lambda)/ (1 - \gamma/\kappa) & \text{for $|z| < \lambda \kappa$} \\
z   & \text{otherwise} \eqsp,
\end{cases}
\end{equation*}
where $S(z,\lambda)= \sign{z} \pospart{|z| -\lambda}$ is the soft thresholding operator (here $\pospart{x}= x \vee 0$ is the positive part of $x$). The proximal operator for $g$ is given by (see \cite[Section~2.1]{parikh2014proximal})
\[
\prox_{\gamma g}(x_1,\dots,x_d)= \left(\prox_{\gamma p_{\lambda,\kappa}}(x_1), \dots, \prox_{\gamma p_{\lambda,\kappa}}(x_d) \right) \eqsp.
\]
The SCAD penalty (\cite{fan2001Variable}) can be handled along the same lines (the expression for the proximal function can be found in \cite[Section~2]{Breheny2011Coordinate}).
\end{example}

\begin{proposition}
\label{prop:stochProxGrad_2-PPAD}
Assume (P\ref{assP:field})--(P\ref{assP:noise}) is satisfied.
 Then, the sequence $\sequence{x}[k][\nset]$ defined by \eqref{stochProxGrad_2} is $K$-PPAD (see \Cref{def_pad}-\eqref{eq:ppad}) with $F= -\nabla f - \bar{\partial} g$, noise sequence $\sequence{\eta}[k][\nset]$ where for each $k \in \nset^*$, $\eta_k= - \delta_k + \{ \nabla f(x_k)  - \nabla f(x_{k-1})\}$ and $y_k =x_k$. Moreover, Assumptions (A\ref{ass:open}--\ref{ass: limit_y_x}) are satisfied and, if $\sup_{k \in \nset^*} \vectornorm{\delta_k} < \infty$  then $\sup_{k \in \nset^*} \vectornorm{\eta_k} < \infty$.
\end{proposition}
\begin{proof}
 We denote by $F$ the set-valued map $- \nabla f - \bar{\partial} g$. The Clarke gradient of a locally Lipschitz function is convex-compact valued and locally bounded (see \cite[Proposition~2.1.2]{clarke1990optimization}) and is upper hemi-continuous (see \cite[Proposition~2.1.5]{clarke1990optimization}. Since  $\nabla f$ is continuous, this implies that  $F$ satisfies (A\ref{ass:open}).
 By \cite[Corollary~2.3.2]{clarke1990optimization},  $\overline{\partial}(-f-g) = -\nabla f - \overline{\partial} g$.

 First, we need to show that $\sequence{x}[k][\nset]$ generated by \eqref{stochProxGrad_2} can be seen as PPAD. Suppose that $\sequence{x}[k][\nset]$ is generated by iteration (\ref{stochProxGrad_2}).  Denote by
 \begin{equation}
 \label{eq:wk}
 w_k = x_{k-1} - \gamma_k \{\nabla f(x_{k-1}) + \delta_k\}\;.\end{equation}
 With this notation \eqref{stochProxGrad_2} can be written as
 \[x_{k} = \prox_{\gamma_k (g + \cvxind_K)} (w_k)\;.\]
By \cite[corollary of Proposition~2.4.3, p.~52]{clarke1990optimization},
$0 \in \gamma_k^{-1} (x_k - w_k) + \bar{\partial} g(x_k) + {\normalcone{K}}(x_k)$.
Therefore there exists
\begin{equation}
\label{eq:definition-u-k}
u_k \in \bar\partial g(x_k)
\end{equation}
such that
\begin{equation}\label{eq:innormalcone}x_k - w_k + \gamma_k u_k \in  - \gamma_k {\normalcone{K}}(x_k)\;.\end{equation}
The normal cone to $K$ at $x_k$ consists of vectors $v_c$, such that for all $z \in K$ we have $\langle v_c , x_k - z \rangle \geq 0$. Since $\gamma_k {\normalcone{K}}(x_k) = {\normalcone{K}}(x_k)$ and using \eqref{eq:wk}, \eqref{eq:innormalcone}  implies that for all $z \in K$,
\begin{equation}\label{ineq:psgi_2}
\pscal{x_k - w_k - \gamma_k u_k}{x_k-z}=\pscal{x_k - x_{k-1} + \gamma_k(\nabla f(x_{k-1}) + \delta_k + u_k)}{x_k - z} \leq 0\;,
\end{equation}
where $u_k$ is defined in \eqref{eq:definition-u-k}.
Setting $v_k = - \nabla f(x_k) - u_k \in F(x_k)$ and $\eta_k = -\delta_k + (\nabla f(x_{k}) - \nabla f(x_{k-1}))$ we get
\[
\pscal{x_k - z}{x_k - x_{k-1} - \gamma_k (v_k + \eta_k)} \leq 0\;
\]
which is \eqref{def:proj_alternative}, but with noise sequence $\sequence{\eta}[k][\nset]$.

We finally have to check that (A\ref{ass:noise}). Since $\sequence{\delta}[k][\nset]$ satisfies condition (P\ref{assP:noise}) and
\begin{equation}
\label{eq:decomposition-eta-k}
\eta_k= \nabla f(x_k)-\nabla f(x_{k-1})-\delta_k= \nabla f(x_k)-\nabla f(x_{k-1})-r^\delta_k - e^\delta_k
\end{equation}
it is enough to show that $\lim_{k \rightarrow \infty} \vectornorm{\nabla f(x_{k-1}) - \nabla f(x_k)} = 0$. Plugging $z=x_{k-1}$ into \eqref{ineq:psgi_2}, we get that:
\[
\pscal{x_{k} - w_k + \gamma_k u_k}{x_k - x_{k-1}}\leq 0\;,
\]
Hence
\begin{align*}
\vectornorm{x_{k} - x_{k-1}}^2 &\leq \pscal{x_k-x_{k-1}-(x_{k} - w_k + \gamma_k u_k)}{x_k - x_{k-1}}\\
&=\pscal{- x_{k-1}+w_k-\gamma_k u_k}{x_k-x_{k-1}}\;,
\end{align*}
and using the Cauchy-Schwartz  and triangle inequalities and \eqref{eq:wk}, we obtain
\begin{equation}\label{eq:bound_x_norm}
\vectornorm{x_{k} - x_{k-1}} \leq \vectornorm{ -x_{k-1} + w_k - \gamma_k u_k}
\leq \gamma_k \vectornorm{\nabla f(x_{k-1})} + \gamma_k \vectornorm{\delta_k} + \gamma_k \vectornorm{ u_k} \;.
\end{equation}
Since $x_{k} \in K$ for any $k \in \nset$ and $\nabla f$ is continuous under (P\ref{assP:field}), $\sup_{k \in \nset} \vectornorm{\nabla f(x_{k-1})} < \infty$. On the other hand, under (P\ref{assP:noise}), $\lim_{k\to \infty} \gamma_k \vectornorm{\delta_k}=0$. Finally, using again (P\ref{assP:field}) and $x_k \in K$ for any $k \in \nset^*$,  \cite[Proposition~2.1.2]{clarke1990optimization} shows that $\sup_{k \in \nset^*} \vectornorm{u_k} < \infty$. Therefore, $\lim_{k \to \infty} \vectornorm{x_k-x_{k-1}}= 0$ under (P\ref{assP:stepsize}) and, the continuity of $\nabla f$ combined with the decomposition \eqref{eq:decomposition-eta-k} shows that (A\ref{ass:noise}) holds.
\end{proof}

\begin{proposition}
\label{prop:stochProxGrad_1-PPAD}
Assume (P\ref{assP:field})--(P\ref{assP:noise}) is satisfied and that $g: \Xset \to \rset$ is Lipschitz.
 Then, the sequence $\sequence{x}[k][\nset]$ defined by \eqref{stochProxGrad_1} is $K$-PPAD with $F= -\nabla f - \bar{\partial} g$ (where $\bar{\partial}$ denotes the Clarke generalized gradient) and noise $\sequence{\eta}[k][\nset]$ where for each $k \in \nset^*$, $\eta_k= - \delta_k + \{ \nabla f(x_k)  - \nabla f(x_{k-1})\}$. Moreover, Assumptions (A\ref{ass:open}--\ref{ass: limit_y_x}) are satisfied and, if $\sup_{k \in \nset^*} \vectornorm{\delta_k} < \infty$  then $\sup_{k \in \nset^*} \vectornorm{\eta_k} < \infty$.
\end{proposition}
\begin{proof}
Denote by
\begin{equation*}
y_k=\prox_{\gamma_k g} \left( x_{k-1} - \gamma_k \nabla f(x_{k-1}) - \gamma_k \delta_k \right)\;.
\end{equation*}
Using the definitions of proximal and \cite[Proposition~2.3.2]{clarke1990optimization}, for each $k \in \nset^*$ there exists $u_{k} \in \bar{\partial} g(y_k)$ such that
\[
y_k= x_{k-1} - \gamma_k \nabla f(x_{k-1}) - \gamma_k \delta_k -\gamma_ku_k\;.
\]
Denoting by $v_k=-\nabla f(y_k) -u_k$ and by $\eta_k= -\delta_k-\nabla f(x_{k-1})+\nabla f(y_k)$ we get that
\[
y_k=x_{k-1}+\gamma_k (v_k+\eta_k)\;.
\]
We now chack (A\ref{ass:noise}). The perturbation $\eta_k$ may be decomposed as $\eta_k= e_k + r_k$ where
$e_k= -e_k^\delta$ and $r_k= -r_k^\delta + \nabla f(y_k) - \nabla f (x_{k-1})$.
Since $\nabla f$ is continuous, Assumption~(A\ref{ass:noise}) is satisfied if
\begin{equation*}
\lim_{k\to\infty} \vectornorm{y_k-x_{k-1}}=0\;.
\end{equation*}
Note that, since $g$ is Lipschitz, \cite[Proposition~2.1.2-(a)]{clarke1990optimization} shows that for all $u \in \bar\partial{g}(y)$, $\vectornorm{u} \leq \lip{g} < \infty$.
Boundedness of $\nabla f$ on the set $K$ and assumptions (P\ref{assP:stepsize}) and (P\ref{assP:noise}) implies that
\[
\limsup_{k \to \infty} \vectornorm{x_{k-1} - y_k} \leq
\limsup_{k \to \infty}\gamma_k \vectornorm{\nabla f(x_{k-1})+\delta_k} +
\limsup_{k \to \infty} \gamma_k \vectornorm{u_k} = 0\;.
\]
Because $x_k$ is a projection of $y_k$ on the set $K$ we have
\[
\vectornorm{y_k -x_{k-1}} \geq \vectornorm{y_k -x_{k}}\;,
\]
showing that (A\ref{ass: limit_y_x}) is satisfied.
\end{proof}

Applying our results from \Cref{sec:convergencePPAD}, we now show that  both versions of the projected proximal gradient algorithms converge.
\begin{theorem} \label{thm:stoch-prox-grad-conv}
Assume (P\ref{assP:field}--\ref{assP:noise}) and denote
\begin{equation}
\label{eq:definition-stationary-point}
\mathcal{S} \eqdef \{ x \in K : 0 \in \nabla f(x) + \bar{\partial} g(x) - {\normalcone{K}}(x) \}\;,
\end{equation}
where $\bar{\partial} g$ is the Clarke gradient of $g$ (see \Cref{def:clarke-gradient}), and ${\normalcone{K}}(x)$ is the normal cone to set $K$ at $x \in K$ (see \eqref{eq:normal_cone}). Suppose $(f+g)(\mathcal{S})$ has empty interior and $\sup_{k \in \nset^*} \vectornorm{\delta_k} < \infty$, where $\delta_k$ is defined in \eqref{eq:definition-delta}.
\begin{enumerate}[(i)]
\item The sequence $\sequence{x}[k][\nset]$ generated by iterations \eqref{stochProxGrad_2} converges  to the $\mathcal{S}$.
\item Assume in addition that $g$ is Lipschitz. Then the sequence $\sequence{x}[k][\nset]$ generated by iterations \eqref{stochProxGrad_1} converges  to the $\mathcal{S}$.
\end{enumerate}
\end{theorem}
\begin{proof}
Using \Cref{prop:stochProxGrad_2-PPAD,prop:stochProxGrad_1-PPAD}, the sequences $\sequence{x}[k][\nset]$ defined by \eqref{stochProxGrad_2} and \eqref{stochProxGrad_1} are $K$-PPAD satisfying (A\ref{ass:open}--\ref{ass: limit_y_x}) with
\begin{equation*}
F= - \nabla f - \bar \partial g \eqsp.
\end{equation*}
To apply \Cref{thm:convergence_proj}, we  show that  $V= f+g$ is a Lyapunov function for
$F_K(x) \eqdef \Pi_{\tangentcone{K}(x)}(F(x))$, $x \in \Xset$.
Under the stated assumptions, $V$ is locally Lipschitz regular and by \cite[Corollary~2 of Proposition~2.3.3]{clarke1990optimization}
\begin{equation*}
\bar\partial V(x)= \nabla f(x) + \bar\partial g(x) \eqsp, \quad x \in \Xset \eqsp.
\end{equation*}
We now  compute the Lie derivative of $V$ with respect to the field $F_K$ see \Cref{def:Lie-derivative}). Let $x \in K$ . Suppose that there exists $a \in \mathcal{L}_{F_K}V(x)$. Then there exists $v \in \Pi_{{\tangentcone{K}}(x)}(F(x))$, such that for all $w \in \bar\partial V(x)$, $\pscal{v}{w} = a$. Let $u \in F(x)$ be such that $\Pi_{{\tangentcone{K}}(x)}(u) = v$. Note that  $-u \in \bar\partial V(x)$, which implies $\pscal{\Pi_{{\tangentcone{K}}(x)} (u)}{-u} = a$ and
\[
a = - \vectornorm{\Pi_{{\tangentcone{K}}(x)} (u)}^2 + \pscal{\Pi_{{\tangentcone{K}}(x)} (u)}{\Pi_{{\tangentcone{K}}(x)}(u)  -u} \eqsp.
\]
Applying \cite[Proposition~0.6.2]{aubin2012differential}, we get that $\pscal{\Pi_{{\tangentcone{K}}(x)} (u)}{ \Pi_{{\tangentcone{K}}(x)}(u)  -u} = 0$. Therefore, if $a \in \mathcal{L}_{F_K}V(x)$, then $a = - \vectornorm{\Pi_{{\tangentcone{K}}(x)} (u)}^2$ for some $u \in F(x)$. Hence, for any $x \in K$, $\mathcal{L}_{F_K}V(x)$  is either empty, or contains only non-positive elements.

For any $x \in \Xset$, \cite[Proposition~2.1.2]{clarke1990optimization} shows that $F(x)$ is non-empty, convex and compact. Define   $U$ for any $x \in K$ as follows
\[
U(x) = - \min_{u \in F(x)} \vectornorm{\Pi_{{\tangentcone{K}}(x)}(u)}^2\;.
\]
By \cite[Proposition~0.6.4]{aubin2012differential} we get for any $x \in K$,
\begin{equation}
\label{eq:representation_U}
U(x) = - \min_{v \in F(x) - {\normalcone{K}}(x)} \vectornorm{v}^2\;.
\end{equation}
Under (A\ref{ass:open}), $F$ is upper hemicontinuous compact-convex valued. On the other hand, by \cite[Chapter 5, Section 1, Theorem 1]{aubin2012differential}, ${\normalcone{K}}$ has closed graph.
Hence, the map $F-N_K$ has closed graph by \cite[Proposition~1.1.2, p.~41]{aubin2012differential}.
By \Cref{lem:lower-semicont-inf}  the function $U$ is upper semicontinuous. Thus we have an upper semicontinuous function $U$, such that for all $x\in K$:
\[
\sup \mathcal{L}_{F_K}V(x) \leq U(x) \leq 0.
\]
Using  \eqref{eq:representation_U}, we get that
\[
\{x\in K\colon U(x) = 0\}=\{x\in K\colon0 \in -\nabla f(x) - \bar\partial g(x) - {\normalcone{K}}(x)\} = \mathcal{S}
\eqsp,
\]
which concludes the proof
\end{proof}

\subsection{Online proximal stochastic gradient descent algorithm}\label{subsec:online}
Let $\Xset$ be an open subset of $\rset^d$ and $K \subset \rset^d$ be a nonempty bounded closed convex set. We  first consider the following composite minimization problem
\begin{equation}
\label{eq:definition-f-online}
\min_{x \in K} \left\{ f(x) + g(x)  \right\}
\end{equation}
where  $f: \rset^d \to \rset$ is continuously differentiable on a neighborhood of $K$ and $g: \Xset \to \rset^d$ is locally Lipschitz, bounded from below, regular function.

In the online learning case, the gradient $\nabla f$ of the function $f$ cannot be computed but that a noisy version of the gradient is available  To make the discussion simple, we assume that
\begin{enumerate}[(i)]
\item an \iid\ sequence $\sequence{\xi}[k][\nset]$ with common distribution $\pi$ on a measurable space $(\Zset,\Zsigma)$
\item there is an oracle which for a given input point $(x,\xi) \in \Xset \times \Zset$ returns a stochastic gradient, a vector $H(x,\xi)$ such that $\PE[\Phi(x,\xi)]= \int_\Zset \Phi(x,z) \pi(\rmd z)$ is well defined and  is equal to $\nabla f(x)$.
\end{enumerate}
We are considering the two following stochastic approximation procedures
\begin{equation}\label{eq:stochProxGrad_sto-2}
x_k = \prox_{\gamma_k (g + \cvxind_K)} \{x_{k-1} - \gamma_k \Phi(x_{k-1},\xi_k)\}
\end{equation}
or
\begin{equation}\label{eq:stochProxGrad_sto-1}
x_k = \Pi_K\left( \prox_{\gamma_k g} \{x_{k-1} - \gamma_k \Phi(x_{k-1},\xi_k)\}) \right) \eqsp.
\end{equation}
where $\sequence{\gamma}[k][\nset]$ is a sequence of step sizes satisfying $\sum_{k=1}^\infty \gamma_k^{1+\epsilon} < \infty$ for some $\epsilon > 0$. Clearly such $\sequence{\gamma}[k][\nset]$ satisfies (P\ref{assP:stepsize}).
Assume for simplicity that the essential supremum of $\sup_{x \in K} \vectornorm{\Phi(x,\cdot)}$ is finite
\begin{equation}
\label{eq:bounded-noise}
\inf \set{a \in \rset}{\pi\left( \sup_{x \in K}\vectornorm{\Phi(x,z)} > a\right)=0} < \infty \eqsp.
\end{equation}
Setting, for all $k\in\nset^*$, $\delta_k= \Phi(x_{k-1},\xi_k)-\nabla f(x_{k-1})$, we get that $\sup_k \vectornorm{\delta_k}<\infty$ $\PP$-almost surely and since $\sequence{\delta}[k][\nset]$ is a bounded martingale increment sequence  the conditional version of the Kolmogorov three-series theorem shows that (see \cite[Theorem~2.16]{hall:heyde:1980} shows that $\sum_{k=1}^\infty\gamma_k\delta_k<\infty$ almost surely. Therefore (P\ref{assP:noise}) holds. Hence, if $(f + g)(\mathcal{S})$ has an empty interior (where $\mathcal{S}$ is the set of stationary point defined in \eqref{eq:definition-stationary-point}), \Cref{thm:stoch-prox-grad-conv} implies that $\sequence{x}[k][\nset]$ generated by \eqref{eq:stochProxGrad_sto-2} almost surely converge to the set of stationary points \eqref{eq:definition-stationary-point}
If in addition the function $g$ is Lipschitz, then the sequence $\sequence{x}[k][\nset^*]$ generated by \eqref{eq:stochProxGrad_sto-1} also converges almost surely to $\mathcal{S}$.
\subsection{Monte Carlo Proximal stochastic gradient descent algorithm}\label{subsec:montecarlo}
In this section, we still consider the composite minimization problem \eqref{eq:definition-f-online}. We assume that $f$ is continuously differentiable and that for all $x \in K$ $\nabla f(x)$ satisfies
  \begin{equation}
    \label{eq:def_H_Markov}
    \nabla f(x)=\int_{\Zset} \Phi(x,z)\pi_x(\rmd z) \eqsp,
  \end{equation}
  for some probability measure $\pi_x$ and an integrable function $(x,z) \mapsto \Phi(x,z)$
  from $K \times \Zset$ to $\rset^d$. Note the dependence of $\pi_x$ on $x$ which makes the situation trikier. To approximate $\nabla f(x)$, several options are
available. Of course, when the dimension of the state space $\Zset$ is small to
moderate, it is always possible to perform a numerical integration using either
Gaussian quadratures or low-discrepancy sequences.  Such approximations necessarily introduce some
bias, which might be difficult to control. In addition, these techniques are
not applicable when the dimension of $\Zset$ becomes large. In
this paper, we rather consider some form of Monte Carlo approximation.

When sampling directly $\pi_x$ is doable, then an obvious choice is to
use a naive Monte Carlo estimator which amounts to sample a batch $\{
\Xb{k}{j}, 1 \leq j \leq m \}$ independently of the past values of the
parameters $\{x_j, j \leq k-1\}$ and of the past draws \ie\  independently
of the $\sigma$-algebra
\begin{equation}
\label{eq:definition-Fn}
\mcf_n \eqdef \sigma(x_0, \Xb{k}{j}, 0 \leq k \leq n, 0 \leq j \leq m) \eqsp.
\end{equation}
We then form
\[
Y_{k}= m^{-1} \sum_{j=0}^{m-1} \Phi(x_{k-1},\Xb{k}{j}) \eqsp.
\]
Conditionally to $\mcf_{k-1}$, $Y_k$ is an unbiased estimator of $\nabla
f(x_{k-1})$.

When direct sampling from $\pi_x$ is not an option, we may still construct a Markov kernel $P_x$ with
invariant distribution $\pi_x$. Monte Carlo Markov Chains (MCMC) provide a
set of principled tools to sample from complex distributions over large
dimensional spaces.  In such case, conditional to the past, $\{ \Xb{k}{j}, 1
\leq j \leq m \}$ is a realization of a Markov chain with transition
kernel $M_{x_{k-1}}$ and started from $\Xb{k-1}{m}$ (the last sample
draws in the previous minibatch).

Recall that a Markov kernel $P$ is an application on $\Zset
  \times \Zsigma$, taking values in $\ccint{0,1}$ such that for any $z
  \in \Zset$, $P(z,\cdot)$ is a probability measure on $\Zsigma$; and
  for any $A \in \Zsigma$, $z \mapsto P(z,A)$ is measurable.
  Furthermore, if $P$ is a Markov kernel on $\Zset$, we denote by
  $P^k$ the $k$-th iterate of $P$ defined recursively as $P^0(z,A)
  \eqdef \1_A(z)$, and $P^k(z,A) \eqdef \int P^{k-1}(z,\rmd z)P(z,A)$,
  $k\geq 1$. Finally, the kernel $P$ acts on probability measure: for
  any probability measure $\mu$ on $\Zsigma$, $\mu P$ is a probability
  measure defined by
\[
\mu P(A) \eqdef \int \mu(\rmd z) P(z,A), \qquad A \in \Zsigma;
\]
and $P$ acts on positive measurable functions: for a measurable
function $f:\Zset \to \rset_+$, $Pf$ is a function defined by
\[
Pf(z) \eqdef \int f(y) \, P(z, \rmd y).
\]
We refer the reader to \cite{meyn:tweedie:2009} for the definitions
and basic properties of Markov chains.

In this section, we assume that $Y_k$ is a Monte Carlo approximation of the expectation $\nabla
f(x_{k-1})$ :
\[
Y_{k} = m^{-1} \sum_{j=0}^{m-1} \Phi(x_{k-1},\Xb{k}{j}) \eqsp;
\]
for all $k \geq 1$, conditionally to the past, $\{ \Xb{k}{j}, 1 \leq j \leq
m \}$ is a Markov chain started from $\Xb{k-1}{m}$ and with transition
kernel $M_{x_{k-1}}$ (we set $\Xb{0}{m} =x_\star \in \Xset$), where for all
$x \in \Xset$, $M_x$ is a Markov kernel with invariant distribution
$\pi_x$.

From a mathematical standpoint, the Markovian setting is trickier than the fixed batch size, because $Y_{k}$
is no longer an unbiased estimator of $\nabla f(x_{k-1})$, \ie\ the bias $B_k$
defined by
\begin{align}
\label{eq:MonteCarlo:bias}
B_k \eqdef  \CPE{Y_{k}}{\mcf_{k-1}} - \nabla f(x_{k-1}) & = m^{-1} \sum_{j=0}^{m-1}
\CPE{ \Phi(x_{k-1},\Xb{k}{j})}{\mcf_{k-1}}  - \nabla f(x_{k-1}) \nonumber  \\
& = m^{-1} \sum_{j=0}^{m-1}  M_{x_{k-1}}^j \Phi(x_{k-1},\Xb{k}{0}) - \nabla f(x_{k-1})  \eqsp,
\end{align}
does not vanish.

Let $W: \Zset \to \coint{1,\infty}$ be a measurable function.
The $W$-norm of a measurable function $h: \Zset \to \rset^\ell$ is defined as
$\normW{h}{W}=\sup_{z\in\Zset} \vectornorm{h(z)}/W(z)$.
The $W$-variation of a finite signed measure $\mu$ on the measurable space $(\Zset,\Zsigma)$ as $\normWm{\mu}{W} \eqdef \sup_{\normW{h}{W} \leq 1} \mu(h)$ where supremum is taken over all measurable functions $h: \Zset \to \rset$ satisfying $\normW{h}{W} \leq 1$.
Denote by $D_W(x,x')$ the $W$-variation of the kernels $M_x$ and $M_{x'}$
\begin{equation*}
D_W(x,x')=\sup_{z \in \Zset}\frac{\normWm{{M_x(z,\cdot)- M_{x'}(z,\cdot)}}{W}}{W(z)}.
\end{equation*}
Consider the following assumptions
\begin{hypB}
 \label{ass: geometric ergodicity}
There exists $\lambda\in \coint{0,1}$, $b<\infty$ and a measurable function $W\colon\Zset\to[1,+\infty)$ such that
 \[
  \sup_{x\in K} \normWm{\Phi(x,\cdot)}{W^{1/2}}<\infty, \quad \sup_{x\in K} M_x W\leq \lambda W +b.
 \]
In addition for any $\ell \in (0,1]$  there exists $C<\infty$ and $\rho\in (0,1)$ such that for any $z\in\Zset$,
\[
 \sup_{x\in K} \normWm{M_x^n(z,\cdot)-\pi_x}{W^\ell}\leq C\rho^n W^\ell(z).
\]
\end{hypB}
Sufficient conditions for the uniform-in-$x$ ergodic
behavior  are given e.g. in \cite[Lemma 2.3]{fort:moulines:priouret:2012}, in terms of
aperiodicity, irreducibility and minorization conditions on the
kernels $\set{M_x}{x \in \Xset}$. Examples of MCMC kernels
$M_x$ satisfying this assumption can be found in
\cite[Proposition 12]{andrieu:moulines:2006}, \cite[Proposition
15]{saksman:vihola:2010}.
\begin{hypB}
 \label{ass:lipschitz-kernel}
 The kernels $M_x$ and the stationary distributions $\pi_x$ are locally Lipschitz \wrt\  $x$, \ie\ for any compact set $K$ and any $x,x'\in K$ there exists $C<\infty$ such that
 \[
  \normWm{\Phi(x,\cdot) -\Phi({x'},\cdot)}{W^{1/2}} + D_{W^{1/2}}(x,x')\leq C \vectornorm{x-x'}\eqsp.
 \]
\end{hypB}

\begin{proposition}\label{prop:markovian}
Let us consider $\sequence{x}[k][\nset^*]$ defined by \eqref{eq:stochProxGrad_sto-2} or \eqref{eq:stochProxGrad_sto-1} with Markovian dynamic $\sequence{z}[k][\nset^*]$ and with $\Phi$ satisfying \eqref{eq:def_H_Markov}.
 Assume (B\ref{ass: geometric ergodicity}--\ref{ass:lipschitz-kernel}), (P\ref{assP:field})--(P\ref{assP:stepsize}), $\PE[W(z_1)]<\infty$, and
 \begin{equation*}
\sum_{k=1}^\infty |\gamma_{k-1}-\gamma_k|<\infty,\quad \sum_{k=1}^\infty \gamma_k^2 <\infty \eqsp.
 \end{equation*}
 Then (P\ref{assP:noise}) is satisfied.
\end{proposition}
\begin{proof}
The proof follows along the same lines as \cite[Proof of Lemma 27]{Miasojedow2013}. However for completeness we give a detailed proof in \Cref{ap:Markov}.
\end{proof}

\subsection{Projected stochastic subgradient descent algorithm}
We consider the projected subgradient descent algorithm framework introduced in \cite{nemirovski2009robust}, for constrained minimization of a possibly nonsmooth convex objective function.
Let $\Xset$ be an open set of $\rset^d$ and $K \subset \Xset$ be a  convex compact. Consider the constrained minimization problem $\argmin_{x\in K } f(x)$, where $f$ is locally Lipschitz
regular function (see~\Cref{def:regular-function}).
The projected stochastic subgradient algorithm generates iteratively the sequence
$\sequence{x}[k][\nset^*]$  as follows
\begin{equation}\label{iter:subgrad}
 x_{k}=\Pi_K(x_{k-1}-\gamma_{k} Y_{k})\eqsp,
\end{equation}
where $\sequence{\gamma}[k][\nset^*]$ is a sequence of positive step sizes, $\Pi_K$ is the projection on the set $K$ and $Y_k$ is a noisy version of Clarke generalized gradient, \ie\
$Y_k=v_k+\delta_k$ with $v_k\in\bar\partial  f(x_{k-1})$.

In \cite{Davis2018Stochastic}, the stochastic subgradient defined recursively by $x_k = x_{k-1}  - \gamma_k Y_k$ (without projection) is analyzed, under the assumption that the iterates $\sequence{x}[k][\nset]$ stay in the compact set $K$ and that the noise $\sequence{\delta}[k][\nset^*]$ is bounded. This paper establishes the almost-sure convergence of the iterates $\sequence{x}[k][\nset^*]$ to the stationary set
$\mathcal{S} \eqdef \set{x \in K}{0 \in \bar\partial{f}(x)}$, under a descent condition on $f$.
Specifically, it is assumed in \cite{Davis2018Stochastic} that if $z : \rset_{+} \rightarrow \rset^d$ is a solution of the differential inclusion $\dot{z}(t) \in - \bar{\partial}f(z(t))$ and $z(0) \not \in \mathcal{S}$ (\ie\ $z(0)$ is not a critical point of $f$), then there exists a $ T > 0$ such that $f(z(T)) < \sup_{t \in [0, T)} f(z(t)) \leq f(z(0))$.
It is also proved in this paper that this condition is satisfied for two classes of functions: subdifferentially regular functions and Whitney stratifiable functions.
The condition of subdifferential regularity is equivalent to our condition of regularity (see \Cref{lem:regular-equivalence}), while the class of Whitney stratifiable functions
is much wider than the class of regular functions, and contains for example the class of semialgebraic and semianalytic functions.

The convergence of stochastic subgradient algorithm for regular functions $f$ can be easily deduced  from \Cref{sec:convergencePAD}. Here, we show that projected stochastic subgradient descent algorithm also fits into our framework and its convergence can be established based on the results of \Cref{sec:convergencePPAD}.

Consider the following assumptions:
\begin{hypPp}
\label{assPp:field}
$\Xset \subset \rset^d$ is an open set, $K\subset\Xset$ is a convex compact set, $f: \Xset \to \rset$ is a
locally Lipschitz, regular function (see \Cref{def:regular-function}).
\end{hypPp}

\begin{hypPp}
\label{assPp:stepsize}
The sequence of step sizes $\sequence{\gamma}[k][\nset^*]$ satisfies $\gamma_k > 0$, $\sum_{k=0}^{\infty} \gamma_k = \infty$, and $\lim_{k \rightarrow \infty} \gamma_k = 0$.
\end{hypPp}
\begin{hypPp}
\label{assPp:noise}
The sequence $\sequence{\delta}[k][\nset^*]$ can be decomposed as $\delta_k= e^\delta_k + r^\delta_k$ where $\sequence{e^\delta}[k][\nset^*]$ and $\sequence{r^\delta}[k][\nset^*]$ are two sequences satisfying $\lim_{k \to \infty} \vectornorm{r^\delta_k} = 0$ and the series $\sum_{k=1}^\infty \gamma_k e^\delta_k$ converges.
\end{hypPp}

\begin{proposition}
\label{prop:stochsubgrad}
Assume (\^{P}\ref{assPp:field}--\ref{assPp:noise}) is satisfied.
 Then, the sequence $\sequence{x}[k][\nset]$ defined by \eqref{iter:subgrad} is $K$-PPAD with $F= - \bar{\partial} f$
 (where $\bar{\partial}$ denotes the Clarke generalized gradient) and noise $\sequence{\eta}[k][\nset]$ where for each $k \in \nset^*$,
 $\eta_k=  -\delta_k $.
 Moreover, Assumptions (A\ref{ass:open}--\ref{ass: limit_y_x}) are satisfied.

\end{proposition}
\begin{proof}
 Set $F= - \bar{\partial} f$. The Clarke gradient of a locally Lipschitz function is convex-compact valued and locally bounded (see \cite[Proposition~2.1.2]{clarke1990optimization}) and is upper hemi-continuous (see \cite[Proposition~2.1.5]{clarke1990optimization}. Hence  (A\ref{ass:open}) is satisfied.
 By construction the $\sequence{x}[k][\nset]$ generated by \eqref{iter:subgrad} is a  K-PPAD with stepsizes $\sequence{\gamma}[k][\nset^*]$, noise $\eta_k=-\delta_k$ and $y_k=x_{k-1}$.
 Hence, assumptions (\^{P}\ref{assPp:stepsize}--\ref{assPp:noise}) implies (A\ref{ass:stepsize}-\ref{ass:noise}). We need to check only (A\ref{ass: limit_y_x}). Denote for any $k\in\nset^*$
 by $w_k=x_{k-1}-\gamma_{k} Y_{k}=x_{k-1}-\gamma_k v_{k}-\gamma_k \delta_k$, where $v_k\in\bar\partial f(x_{k-1})$. Since $x_k$ is projection of $w_k$ on the compact convex set $K$, the triangle inequality implies
 \begin{equation}
  \label{eq:y_x_sub}
 \vectornorm{x_{k}-y_{k}}=\vectornorm{x_{k}-x_{k-1}}\leq\vectornorm{w_{k}-x_{k-1}}\leq \gamma_k\vectornorm{v_k}+ \gamma_k\vectornorm{\delta_k}\eqsp.
 \end{equation}
Since $\sequence{x}[k][\nset^*]$ remains in the compact set $K$ and $\bar\partial f$ is localy bounded we obtain that $\sup_k \vectornorm{v_k}<\infty$. Thereofore by (\^{P}\ref{assPp:stepsize}--\ref{assPp:noise}) and
\eqref{eq:y_x_sub} we get that $\lim_{k\to\infty}\vectornorm{x_{k}-y_{k}}=0$, and that completes the proof.

 \end{proof}

Applying our results from \Cref{sec:convergencePPAD}, we now show that  projected stochastic subgradient algorithms converge.
\begin{theorem} \label{thm:stoch-prox-subgrad-conv}
Assume (\^{P}\ref{assP:field}--\ref{assP:noise}) and denote
\begin{equation*}
\mathcal{S} \eqdef \set{x \in K}{0 \in  \bar{\partial} f(x) - {\normalcone{K}}(x)} \eqsp,
\end{equation*}
where $\bar{\partial} f$ is the Clarke gradient of $g$ (see \Cref{def:clarke-gradient}), and ${\normalcone{K}}(x)$
is the normal cone to set $K$ at $x \in K$ (see \eqref{eq:normal_cone}). Suppose $f(\mathcal{S})$ has empty interior
and $\sup_{k \in \nset^*} \vectornorm{\delta_k} < \infty$.
Then the sequence $\sequence{x}[k][\nset]$ generated by the iterations \eqref{iter:subgrad} converges  to the $\mathcal{S}$.
\end{theorem}
\begin{proof}
 The proof follows along the sime lines as proof of \Cref{thm:stoch-prox-grad-conv}.
\end{proof}
Note, that adaptation of results from \Cref{subsec:online} and \Cref{subsec:montecarlo} to the case of projected stochastic subgradient descent algorithm is straightforward.

\section{Proofs}\label{ap:prelim}
In this section we introduce some notations and preliminary facts used in the proofs of results from \Cref{sec:convergencePAD} and \Cref{sec:convergencePPAD},
as well as some auxiliary definitions and theorems.

\begin{lemma}\label{lem:lower-semicont-inf}
Let $K$ be a compact subset of $\rset^d$, and
$G: K \rightarrow \mathcal{P}(\rset^d)$ be a nontrivial closed set-valued map (\ie\ for any $x \in K$, $G(x) \ne \emptyset$ and the graph of the function $G$ is closed; see \cite[Section~2]{clarke1990optimization}).
Then  $W : K \rightarrow \rset$ defined by
\[
W(x) = -\min_{v \in G(x)} \Vnorm{v}^2
\]
is upper semicontinuous.
\end{lemma}
\begin{proof}
Let $x \in K$ and $\sequence{x}[n][\nset] \subset K$  be any sequence such that $\lim_{n\to\infty}x_n=x$. Consider a subsequence $\subsequence{x}[n][k][\nset]$ such that
$\lim_k{W(x_{n_k})}=\limsup_n W(x_n)$.
For any $\tilde{x} \in K$, there exists $\tilde{w} \in G(\tilde{x})$ such that $-\vectornorm{\tilde{w}}^2= W(\tilde{x})$ (since $G(\tilde{x})$ is closed and $\|\cdot\|$ is continuous).
We may therefore define a sequence $\subsequence{w}[n][k][\nset]$  such that for all $k \in \nset$, $-\vectornorm{w_{n_k}}^2=W(x_{n_k})$. Because $\lim_{k \to \infty} W(x_{n_k})$ is finite, the subsequence $\subsequence{w}[n][k][\nset]$ is bounded. We may hence  extract a subsequence
$\subsequence{w}[\tilde n][k][\nset]\subseteq \subsequence{w}[n][k][\nset]$ such that $\lim_{k\to\infty} w_{\tilde{n}_k} = w$ for some $w\in \rset^d$. Since $G$ has closed graph we have
$w\in G(x)$. By continuity of norm we get
\[
 \limsup_n{W(x_n)}=\lim_{k\to\infty} W(x_{\tilde n_k})=-\lim_{k\to\infty} \vectornorm{w_{\tilde n_k}}^2=-\vectornorm{w}^2\leq W(x)\eqsp,
\]
which means that $W$ is upper semicontinuous.
\end{proof}

\begin{lemma}
\label{lem:regularity-sum}
Let $d \in \nset^*$,  $p_i: \rset \to \rset$, $i \in \{1,\dots,d\}$,  be functions and $x^0= (x^0_1,\dots,x^0_d) \in \rset^d$. If for any $i \in \{1,\dots,d\}$ the functions $p_i$ are regular at $x^0_i$, then the function $g(x_1,\dots,x_d)= \sum_{i=1}^d p_i(x_i)$ is regular at $x^0$.
\end{lemma}
\begin{proof}
By \cite[Proposition~2.3.6]{clarke1990optimization} a finite linear combination (by nonnegative scalars) of functions regular at $x^0$ is regular at $x^0$. The proof then follows by noting that, for any $i \in \{1,\dots,d\}$, the function $p^i(x_1,\dots,x_d)= p_i(x_i)$ is regular at $x^0$.
\end{proof}
\begin{lemma} \label{lem:regular-equivalence}
Let $f : \mathbb{R}^d \rightarrow \mathbb{R}$ be a locally Lipschitz function. Then $f$ is regular at $x$ if and only if for all $v \in \bar{\partial} f(x)$ we have:
\begin{equation} \label{property:subdiff-reg}
f(y) \geq f(x) + \pscal{v}{y-x} + o(\vectornorm{x - y}) \eqsp.
\end{equation}
\end{lemma}
\begin{proof}
Let $\partial_D f(x)$ be the D-subdifferential of $f$ at $x$ (see \cite[chapter~3.4, subsection D-differential]{clarke1998nonsmooth}  for definition). Then according to \cite[Proposition~4.10]{clarke1998nonsmooth} , the inequality \eqref{property:subdiff-reg} is satisfied for $v$ if and only if $v \in \partial_D f(x)$. On the other hand, \cite[Proposition~4.8, part~(b)]{clarke1998nonsmooth} implies that for a locally Lipschitz function we have $\bar{\partial} f(x) = \partial_D f(x)$ if and only if $f$ is regular. Combined, these two facts conclude the proof.
\end{proof}

\begin{definition}[Equicontinuity]
\label{def:pointwise-equicont}
A sequence of functions $\sequence{f}[n][\nset]$, from $\rset$ to $\rset^d$, is said
to be equicontinuous at $t_0$, if for all $\epsilon > 0$ there exists $\eta>0$
such that for all $|t - t_0| \leq \eta$,  $\Vnorm{f_n(t) -  f_n(t_0)} \leq  \epsilon$ for all $n \in \nset$.  A sequence of functions $\sequence{f}[n][\nset]$ from
$\rset$ to $\rset^d$ is said to be  equicontinuous, if and only if it is
equicontinuous at every point of $t_0 \in \rset$.
\end{definition}

\begin{theorem}[Arzela-Ascoli theorem]
\label{thm:Arzela-Ascoli}  \\
  Let $\sequence{f}[n][\nset]$ from $\rset$
  to $\rset^k$ be a sequence of functions. Assume that the sequence $\sequence{f}[n][\nset]$ is equicontinuous and pointwise bounded (meaning that
  $\sup_{n \in \nset} \vectornorm{f_n(x)}$ is finite for all $x \in \rset$).
  Then the sequence $\sequence{f}[n][\nset]$  is precompact in the topology of compact convergence.

\end{theorem}
\begin{proof}
See \cite[Theorem 4.44]{FollandAnalysis}.
\end{proof}

\begin{lemma} \label{lem:weak_conv}
Let $\sequence{f}[n][\nset]$ be a sequence of functions from $\rset$ to $\rset^d$, that are  absolutely continuous on compact intervals,
and converge pointwise to a function $f$, which is also  absolutely continuous on compact intervals.
For each $n \in \nset$, let $g_n$ be a weak derivative of $f_n$ and $g$ be a weak derivative of $f$.
Also assume that $\sequence{g}[n][\nset]$ are uniformly integrable on bounded intervals.
Then for every interval $[a,b]$, $0 \leq a < b <\infty$, the sequence $\{ g_n \indi{[a,b]}, n \in \nset\}$  converges in the weak topology of $L_1([a,b])$ to $g \indi{[a,b]}$, \ie\ we get, for all $\varphi \in L_\infty([a,b])$,
\[
\lim_{n \to \infty} \int_a^b g_n(t) \varphi(t) \rmd t = \int_a^b g(t) \varphi(t) \rmd t \eqsp.
\]
\end{lemma}

\begin{proof}
We prove the result for  $d=1$, the extension for $d>1$ is straightforward. Denote by
$h_n = \mathbbm{1}_{[a,b]} g_n$, and $h = \mathbbm{1}_{[a,b]} g$. Under the stated assumptions, for each $n \in \nset$,    $h_n \in L_1([a,b])$ and the sequence $\sequence{h}[n][\nset]$ is uniformly integrable.
Using the Dunford-Pettis theorem \cite[Corollary 4.7.19]{Bogachev2006}, we conclude that for every subsequence $\sequence{h}[n][\nset]$ there exists a further
subsequence $\subsequence{h}[n][l][\nset]$ which converges in the weak topology of $L_1[a,b])$ to $h_*$.
We know that for any $c$ such that $a \leq c \leq b$ we have:
\[
\int_a^c h_* dx = \lim_{l \rightarrow \infty} \int_a^c h_{n_l} dx = \lim_{l \rightarrow \infty} (f_{n_l}(c) - f_{n_l}(a)) = f(c) - f(a)
\]
Hence $h_*$ is a weak derivative of $f$ (restricted to $[a,b]$). Since a weak derivative is unique up to a set of measure zero, we conclude that $h_* = h$ in  $L_1([a,b])$.
Since from every subsequence of $\sequence{h}[n][\nset]$ we can choose a further subsequence converging weakly to $h$, the weak limit of $\sequence{h}[n][\nset]$ exists and is equal to $h$. This concludes the proof.
\end{proof}

\begin{definition}[Convex combination subsequence]\label{def:convex_combination}
Let $\sequence{z}[n][\nset]$ be a sequence belonging to a linear subspace over $\rset$.
Let $\{w_{n,k}, k \in \nset, n \in \nset\}$ be a sequence of weights satisfying:
\begin{enumerate}[(i)]
\item For all $k,n \in \nset$,  $0 \leq w_{n,k} \leq 1$.
\item For all $n\in \nset$ we have $\sum_{k=1}^{\infty} w_{n,k} = 1$, $w_{n,k} > 0$ only for a finite number of indices $k$, and
\[
\lim_{n \rightarrow \infty} \left( \inf \{ k: w_{n,k} > 0 \} \right) = \infty
\]
\end{enumerate}
The sequence  $\sequence{z^w}[n][\nset]$ defined for each $n \in \nset$ by
$z^w_n= \sum_{k=1}^{\infty} w_{n,k}z_k$
is said to be a convex combination of $\sequence{z}[n][\nset]$ with weights $\{w_{n,k}, k \in \nset, n \in \nset \}$.
\end{definition}
In the sequel, we will consider convex combinations of elements of $\rset^d$ and of the Banach space of integrable functions over some intervals $\ccint{a,b}$ of $\rset$, $L_1(\ccint{a,b})$.
\begin{lemma} \label{lem:L1-conv-comb}
Let $0 \leq a < b <\infty$. Let $\sequence{f}[n][\nset]$ be a sequence of functions from $[a,b]$ to $\rset^d$ which are integrable.  Suppose the sequence $\sequence{f}[n][\nset]$
converges in the weak topology of $L_1([a,b])$ to a limit $f$. Then there exists a convex combination subsequence of $\sequence{f}[n][\nset]$ that converges almost everywhere to $f$.
\end{lemma}

\begin{proof}
The space $L_1([a,b])$  is a Banach space, so from Mazur's Lemma \cite[Corollary 3.8]{Brezis2010} it follows that there exists a convex combination subsequence of
$\sequence{f}[n][\nset]$ that converges strongly in $L_1([a,b])$ to $f$, \ie\ there exits a sequence of weights $\{w_{n,k}, n,k \in \nset\}$ satisfying the conditions of \Cref{def:convex_combination} such that
\[
\lim_{n \to \infty} \int_a^b |f^w_n(t) - f(t)| \rmd t = 0 \eqsp.
\]
A strongly convergent sequence in $L_1([a,b])$ has an almost everywhere convergent subsequence.
Since a subsequence of a convex combination subsequence is a convex combination subsequence,
it follows that $\sequence{f}[n][\nset]$ has a convex combination subsequence that converges almost everywhere to $f$.
\end{proof}

\begin{lemma}\label{lem:convex_lim}
Let $G: \Xset \to \mathcal{P}(\rset^d)$ be a set-valued function, define on a open subset $\Xset \subset \rset^d$.
Assume that $G$  is upper hemicontinuous and convex-closed valued. Let $\sequence{x}[k][\nset] \subset \Xset$ be a sequence that converges to $x \in \Xset$, and $\sequence{v}[k][\nset]$
be a sequence such that $v_k \in G(x_k)$ for any $k \in \nset$. Suppose a convex combination subsequence of $\sequence{v}[k][\nset]$ converges to a limit $v$. Then $v \in G(x)$.
\end{lemma}
\begin{proof}
 For a closed set $A\subseteq\rset^d$ and  $z\in\rset^d$ we denote by
$\dist(z,A)=\inf_{y\in A}\vectornorm{z-y}$.
Let $\sequence{v^w}[k][\nset]$ be a convex combination subsequence of $\sequence{v}[k][\nset]$ with weights $\{ w_{n,k}, n,k \in \nset\}$ (see \Cref{def:convex_combination}).
Since $G$ is upper hemicontinuity (see \Cref{def:upper-hemicontinuity}), for any $\epsilon > 0$,  there exists an integer $k_\epsilon$, such that for $k \geq k_\epsilon$, $\dist(v_k, G(x))\leq \epsilon$.
Since $G(x)$ is convex, then function $\dist(\cdot, G(x))$ is convex, and hence from Jensen's inequality we have
$\dist(v^w_n, G(x)) \leq \epsilon$ for $n$ large enough.
Hence for any $\epsilon>0$ and $n$ large enough we have $\dist(v,G(x))\leq\vectornorm{v-v_n^w}+\dist(v^w_n,G(x))\leq2\epsilon$.
\end{proof}

\bibliographystyle{siamplain}
\bibliography{ConvStochProx}

\appendix
\section{Proof of \Cref{prop:markovian}}\label{ap:Markov}

\begin{lemma}\label{lem:lipschitz x} Under assumptions of \Cref{prop:markovian}, there exists  $C<\infty$ such that for any $k\in\nset^*$,
 \[\vectornorm {x_k -x_{k-1}} \leq C\gamma_{k} W^{1/2}(z_k)\]
\end{lemma}
\begin{proof}
First consider $\sequence{x}[k][\nset^*]$ generated by \eqref{eq:stochProxGrad_sto-2}. By \eqref{eq:bound_x_norm} we get
\[
 \vectornorm{x_k-x_{k-1}}\leq \gamma_k \vectornorm{\nabla f(x_{k-1})} + \gamma_k \vectornorm{\delta_k} + \gamma_k \vectornorm{ u_k}\eqsp,
\]
where $\delta_k= \Phi(x_{k-1},z_k)-\nabla f(x_{k-1})$ and $u_k\in\bar\partial g(x_k)$. Since  $\sequence{x}[k][\nset^*]\subset K$, where $K$ is compact, by (P\ref{assP:field}) we obtain
$\sup_{k\in\nset^*}\{\vectornorm{u_k}+ \vectornorm{\nabla f(x_{k-1}}\}<\infty$. In addition by (B\ref{ass:lipschitz-kernel}), there exists $\tilde C<\infty$
such that for any $z \in \Zset$, $\sup_{k\in\nset^*}\vectornorm{\Phi(x_{k-1},z)}\leq \tilde C W^{1/2}(z)$ which concludes the proof.

Now consider $\sequence{x}[k][\nset^*]$ generated by \eqref{eq:stochProxGrad_sto-1}. Denote by
\begin{equation*}
y_k=\prox_{\gamma_k g} \left( x_{k-1} - \gamma_k \nabla f(x_{k-1}) - \gamma_k \delta_k \right)\eqsp,
\end{equation*}
where $\delta_k= \Phi(x_{k-1},z_k)-\nabla f(x_{k-1})$.
Using the definitions of the proximal operator (see \eqref{eq:proximal}) and \cite[Proposition~2.3.2]{clarke1990optimization}, for each $k \in \nset^*$ there exists $u_{k} \in \bar{\partial} g(y_k)$ such that
\[
y_k= x_{k-1} - \gamma_k \nabla f(x_{k-1}) - \gamma_k \delta_k -\gamma_ku_k\;.
\]
Since $x_k$ is projection of $y_k$ on set $K$, by the triangle inequality we get
\[
 \vectornorm{x_k-x_{k-1}}\leq  \vectornorm{y_k-x_{k-1}}\leq \gamma_k \vectornorm{\nabla f(x_{k-1})} + \gamma_k \vectornorm{\delta_k} + \gamma_k \vectornorm{ u_k}\eqsp.
\]
Note that, since $g$ is Lipschitz, \cite[Proposition~2.1.2-(a)]{clarke1990optimization} shows that for all $u \in \bar\partial{g}(y)$, $\vectornorm{u} \leq \lip{g} < \infty$.
Boundedness of $\nabla f$ on the set $K$ and assumption (B\ref{ass:lipschitz-kernel}) concludes the proof.$\ $
\end{proof}
\begin{proof}[Proof of \Cref{prop:markovian}]
In this proof, $C$ is a constant whose value may change upon each appearance.
 Observe that it is enough to show $\sum_{k=1}^\infty \gamma_k\delta_k<\infty$ almost surely, where by construction
 \[
 \delta_k = \Phi({x_{k-1}},z_k)-\nabla f(x_{k-1})=\Phi({x_{k-1}},z_k)-\pi_{x_{k-1}}(\Phi({x_{k-1}},\cdot))\eqsp.
 \]
Geometric ergodicity (B\ref{ass: geometric ergodicity}) in turn implies the existence of a solution of
the Poisson equation, and also provide bounds on the growth of this
solution; see \cite[Lemma~13]{atchade2017Perturbed}. For $z \in \Zset$, we set $\Phi_x(z)= \Phi(x,z)$.
For any $x\in K$ there exists a solution $\hat \Phi_x$ to the Poisson equation
\[
\hat\Phi_x-M_x \hat\Phi_x =\Phi_x-\pi_x(\Phi_x)
\]
and there exists a constant $C<\infty$ such that for any $x \in K$ and  $z \in \Zset$
\begin{equation}\label{eq:bound poissson norm}
\normW{\hat \Phi_x(z)}{W^{1/2}}\leq C W^{1/2}(z) \quad \text{and} \quad \normW{M_x \hat{\Phi}_x(z)}{W^{1/2}}\leq C W^{1/2}(z) \eqsp.
\end{equation}
Hence, for any $k\in\nset^*$ we can decompose $\delta_k = \delta M_k+\kappa_k$ where
\begin{align*}
 \delta M_k &\eqdef \hat \Phi_{x_{k-1}}(z_k)-M_{x_{k-1}}\hat \Phi_{x_{k-1}}(z_{k-1})\\
 \kappa_k   &\eqdef M_{x_{k-1}}\hat \Phi_{x_{k-1}}(z_{k-1}) -M_{x_{k-1}}\hat \Phi_{x_{k-1}}(z_{k}).
\end{align*}
By construction, the sequence $\sequence{\delta M}[k][\nset]$ is a martingale increment sequence and Doob's inequality implies that there exists a constant $C < \infty$ such that for all $z_0 \in \Zset$,
\begin{equation}\label{eq:doob}
 \PE \left[\left(\sup_{k>n}\left|\sum_{l=n}^k\gamma_l\delta M_l\right|\right)^2 \right]
 \leq C \sum_{l=n}^\infty \gamma_l^2 \PE \left[\left\|\hat \Phi_{x_{l-1}}(z_l)-M_{x_{l-1}}\hat \Phi_{x_{l-1}}(z_{l-1})\right\|^2\right]\eqsp.
\end{equation}
By construction the sequence $\sequence{x}[k][\nset]$ remains in the compact set $K$. Using \eqref{eq:bound poissson norm}, we obtain that,
\[
\PE \left[\left(\sup_{k>n}\left|\sum_{l=n}^k\gamma_l\delta M_l\right|\right)^2 \right] \leq C \sup_{k \geq 1} \PE[W(z_k)] \sum_{l=n}^\infty \gamma_l^2 .
\]
Geometric ergodicity together with $\mathbb{E}W(z_1)<\infty$ implies that (see.
\cite[Lemma~21]{Miasojedow2013} or \cite[Lemma~14]{atchade2017Perturbed})
\[
 \sup_k \mathbb{E} W(z_k)<\infty.
\]
Therefore, since $\sum_{k=1}^\infty \gamma_k^2<\infty$  we conclude that $\sum_k\gamma_k\delta M_k$ converges almost surely.

Decompose $\sum_{l=n}^k \gamma_l \kappa_l = R_{n,k}^1+R_{n,k}^2+R_{n,k}^3$ with
\begin{align*}
 R_{n,k}^1&\eqdef\sum_{l=k-1}^{n-1} \gamma_{l+1}\left[M_{x_l}\hat \Phi_{x_l}(z_l)-M_{x_{l-1}}\hat \Phi_{x_l}(z_l)  \right]\\
 R_{n,k}^2&\eqdef\gamma_{n-1}M_{x_{n-2}}\hat \Phi_{x_{n-2}}(z_{n-1})-\gamma_{k}M_{x_{k-1}}\hat \Phi_{x_{k-1}}(z_{k})\\
 R_{n,k}^3&\eqdef \sum_{l=n-1}^{k-1}(\gamma_{l+1}-\gamma_{l})M_{x_{l-1}}\hat \Phi_{x_{l-1}}(z_l)
\end{align*}

Applying \cite[Lemma 4.2]{fort:moulines:priouret:2012} we get that
\[|R_{n,k}^1|\leq C \sum_{l=k-1}^{n-1} \gamma_{l+1}W^{1/2}(z_l)\left[ D_{W^{1/2}}(x_l,x_{l-1})+\Vert \Phi({x_l},\cdot)- \Phi({x_{l-1}},\cdot)\Vert _{W^{1/2}} \right]\]
By assumption (B\ref{ass:lipschitz-kernel}) and \Cref{lem:lipschitz x} we obtain $\sup_k [\PE|R_{n,k}^1|]\leq C  \sup_k \PE[W(z_k)] \sum_{l>n} \gamma_l^2$  and this converges to zero by $\sum_{k=1}^\infty \gamma_k^2<\infty$ .

Finally, from \eqref{eq:bound poissson norm}, (P\ref{assP:stepsize}) and $\sum_k|\gamma_k-\gamma_{k-1}|<\infty$ we deduce that $R_{n,k}^2$ and $R_{n,k}^3$ also converges to zero almost surely
and that completes the proof.
$\ $
\end{proof}

\end{document}